\documentclass[preprint,review]{elsarticle}

\usepackage{lineno}

\usepackage{amsmath,amssymb}
\usepackage{stmaryrd}
\usepackage{enumerate}
\usepackage{algorithm}
\usepackage[noend]{algpseudocode}
\usepackage{calrsfs}

\usepackage{graphicx}
\usepackage{amssymb}
\usepackage{array}
\usepackage{amsmath}
\usepackage{amsthm}
\usepackage{epstopdf}
\usepackage{epsfig}
\usepackage{subfigure}
\usepackage{pstricks}
\usepackage{fancyheadings}

\usepackage{stackengine}

\usepackage[margin=1.0in]{geometry}
\DeclareMathAlphabet{\pazocal}{OMS}{zplm}{m}{n}

\usepackage{graphicx}

\usepackage{stackengine}

\makeatletter
\def\BState{\State\hskip-\ALG@thistlm}
\makeatother

\newcommand{\lJump}{[\![}
\newcommand{\rJump}{]\!]}

\newtheorem{theorem}{Theorem}

\newtheorem{remark}{Remark}
\newtheorem{definition}{Definition}

\begin{document}
\begin{frontmatter}
\title{A new discontinuous Galerkin  method  for elastic waves with physically motivated numerical fluxes}
\author[label1,label2,label4]{Kenneth Duru}
\author[label3]{Leonhard Rannabauer }
\author[label2]{Alice-Agnes Gabriel}
\author[label2]{Heiner Igel}

 \address[label1]{Mathematical Sciences Institute, The Australian National University, Canberra, Australia}
  \address[label2]{Department of Geophysics, Ludwig-Maximilian University, Munich, Germany}
   \address[label3]{Technical University of Munich, Germany}
  \address[label4]{Corresponding author: kenneth.duru@anu.edu.au}

  \pagenumbering{arabic}
\begin{abstract}

The discontinuous Galerkin  (DG) method  is an established method for computing  approximate solutions of partial differential equations in many applications. Unlike continuous finite elements, in DG methods, numerical fluxes  are  used to enforce inter-element  conditions, and internal and external physical boundary conditions.  However, for certain problems   such as elastic wave propagation in complex media,  where several wave types and wave speeds are simultaneously present, a standard numerical flux may not be compatible with the physical boundary conditions. If surface or interface waves are present, this incompatibility may lead to numerical instabilities.  We present a stable and arbitrary order accurate DG method for elastic waves with a physically motivated numerical flux. Our numerical flux is compatible with all well-posed, internal and external, boundary conditions, including linear and nonlinear  frictional constitutive equations for modelling spontaneously propagating shear ruptures in elastic solids and dynamic earthquake rupture processes. 

First, we generate boundary or interface data by solving a Riemann-like problem constrained against  the physical conditions acting at internal or external element boundaries.
Second, we penalise the data on the boundary against incoming characteristics. Third, we construct a flux fluctuation vector obeying the eigen-structure of the underlying PDE. Finally, we append the flux fluctuation vector to the discretized PDE with physically motivated penalty weights.

 By construction our choice of penalty parameters yield an upwind scheme and a discrete energy estimate analogous to the continuous energy estimate. The  spectral radius of the resulting spatial operator has an upper bound which is independent of the boundary and interface conditions, thus it is suitable for efficient explicit time integration. We present numerical experiments in one and two space dimensions verifying high order accuracy and asymptotic numerical stability, and demonstrating potentials for modelling complex nonlinear frictional problems in elastic solids.
\end{abstract}

\end{frontmatter}
\section{Introduction}\label{sec:introduction}
High order accurate and explicit time-stable solvers are well suited  for hyperbolic wave propagation problems. See, for example, the pioneering work by Kreiss and Oliger \cite{KreissOliger1972}.   However, because  of the complexities of real geometries, internal interfaces, nonlinear boundary/interface conditions and the presence of disparate spatial and temporal scales present in real media and sources, discontinuities and sharp wave fronts become  fundamental features of the solutions. Thus, in addition to high order accuracy, geometrically flexible and adaptive numerical algorithms are critical for high fidelity and efficient simulations of wave phenomena in many applications. 
The discontinuous Galerkin method (DG method)  has been demonstrated to posses  the desirable properties needed to effectively simulate wave phenomena occurring in geometrically complex and heterogeneous media \cite{HesthavenWarburton2002, delaPuenteAmpueroKaser2009,  PeltiesdelaPuenteAmpueroBrietzkeKaser2012}. Since its introduction \cite{ReedHill1973}, the DG method has been developed and analyzed for hyperbolic partial differential equations (PDEs), see for examples \cite{CockburnShu1989}--\cite{CockburnHouShu1990}, \cite{HesthavenWarburton2008}--\cite{HesthavenWarburton2002} and the references therein. The DG method combines ideas from high order finite element methods with traditional finite volume and finite difference methods, yielding local discrete operators with spectral accuracy. 
The power of DG method lies in the local nature of the spatial operators with high order accuracy, and the flexibility of the method for resolving complex geometries using unstructured and/or boundary conforming curvilinear meshes \cite{delaPuenteAmpueroKaser2009, PeltiesdelaPuenteAmpueroBrietzkeKaser2012, KoprivaGassner2014, Warburton2013, DumbserKaser2006}.  
Because of the spatial locality of the operators, DG method easily lends itself to efficient parallel numerical algorithms on modern heterogeneous high performance computing platforms \cite{Heineckeetal2014, BursteddeWilcoxGhattas2011}. DG method has been successfully applied to a variety of applied mathematics problems, and in particular to wave propagation and computational fluid dynamics  problems \cite{DumbserPeshkovRomenski, KirbyKarniadakis2005}.

 In the past decade,  DG method has   gained popularity in engineering and applied sciences, and it is increasingly becoming  attractive, as a method of choice for computing approximate solutions of PDEs in academia and industry. However, wave propagation problems often appear with nontrivial boundary conditions that are not covered by standard DG method methods. Examples include linear and nonlinear friction laws, describing earthquake rupture physics,  nonlocal transparent boundary conditions, local absorbing boundary conditions, and other dynamic boundary conditions that result from local or nonlocal coupling with differential equations on the boundary. 

In the current work, we continue the effort to develop and analyze DG method, focusing on seismological applications. We are particularly interested in reliable numerical modeling of nonlinear earthquake source processes and  high fidelity simulations of elastic waves in heterogeneous and geometrically complex solid Earth models.
Seismic waves emanating from geophysical events propagate over hundreds to thousands of kilometers interacting with tectonic forces, geological structure, complicated topography and earthquake source processes on scales down to millimeters. Exploration seismology and natural earthquake hazard mitigation increasingly rely on multi-scale (0--20 Hz) and multi-physics (non-linear rheology, fluid and heat transport, dynamic rupture sources) simulations.  The fracture mechanical description of non-linear frictional failure (dynamic rupture) on a pre-defined fault can be treated as an internal boundary condition \cite{KanekoLapustaAmpuero2008, DuruandDunham2016, delaPuenteAmpueroKaser2009, PeltiesdelaPuenteAmpueroBrietzkeKaser2012}. Non-linear boundary conditions and material behavior may lead to very large gradients in the numerical solution. Accurate and efficient numerical simulation of these problems require carefully designed and provably stable numerical methods.

The DG method has been successfully applied to solve the elastic wave equation, including (elementwise constant) heterogeneous material properties \cite{delaPuenteAmpueroKaser2009, PeltiesdelaPuenteAmpueroBrietzkeKaser2012, DumbserKaser2006}.
However, a crucial component of DG method is the numerical flux \cite{DeGraziaMengaldoMoxeyVincentSherwin2013, Huynh2007}, inherited from finite volume and finite difference methods \cite{Godunov1959, Rusanov1961} for hyperbolic PDEs, based on approximate or exact solutions of the Riemann problem.   It is rather not surprising that high order flux reconstruction finite volume methods \cite{DeGraziaMengaldoMoxeyVincentSherwin2013, Huynh2007} have been shown to be analogous to the DG method. Once the  solution of the  Riemann problem is available,  information  is exchanged across the element boundaries using numerical fluxes.    The Rusanov flux \cite{Rusanov1961} (also called local Lax-Friedrichs flux) is widely used,  because of its simplicity and robustness. Other numerical fluxes such as the centered flux, Godunov flux, Roe flux, and the Engquist-Osher flux, have also been used. The choice of a numerical flux is critical for accuracy and stability of the DG method \cite{Qiu2008, KirbyKarniadakis2005, KoprivaNordstromGassner2016}. For example, including nonlinear frictional models by direct adaption of a Godunov flux introduces a very selective numerical dissipation avoiding spurious high-frequency oscillations which can be problematic in many other solvers of dynamic earthquake rupture and seismic wave propagation \cite{delaPuenteAmpueroKaser2009, PeltiesdelaPuenteAmpueroBrietzkeKaser2012}.   This is due to the upwind property of the Godunov flux, which has been corroborated in the recent paper, \cite{KoprivaNordstromGassner2016}, elucidating the benefits of an upwind flux over a centered flux for first order hyperbolic problems.  However,  issues of normal stress  inconsistency and instability have been reported, when incorporating  nonlinear frictional models in DG method using standard numerical fluxes,  such as the Godunov flux. Thus for problems where interesting linear/nonlinear physical phenomena occur at internal and external boundaries there is a need to develop numerical fluxes that obey the underlying physics.

 For elastic wave propagation in complex media, and where several wave types and wave speeds are simultaneously present, a numerical flux may not be compatible with physical boundary conditions.  In particular, if surface or interface waves are present, this incompatibility can lead to (longtime) numerical instabilities which will eventually destroy the accuracy of numerical simulations. Our preliminary numerical studies show that the Rusanov flux \cite{Rusanov1961}  exhibits  numerical instability when Rayleigh surface waves are present.   In this study we develop a new DG  flux incorporating the physical conditions acting at the element boundaries. 
 The new physically motivated numerical  flux is designed to be compatible with all well-posed and energy stable physical boundary conditions, including linear and nonlinear friction laws, modeling earthquake rupture dynamics \cite{KanekoLapustaAmpuero2008, DuruandDunham2016, PeltiesdelaPuenteAmpueroBrietzkeKaser2012}.  
 
 The main objective of this initial paper is to formulate an alternative way to couple DG elements in elastic solids using physical conditions, with rigorous mathematical support. Our fundamental idea is to use friction to glue DG elements together, in elastic solids, in a provably stable manner. To the best of our knowledge, this has never been reported before in the literature. Thus, all DG inter-element interfaces are frictional interfaces with associated frictional strength.  Classical inter-element interfaces where slip is not permitted have infinite frictional strength, and can never be broken by any load of finite magnitude. Other interfaces where frictional  slip are accommodated have finite frictional strength, and are governed by a generic friction law \cite{Scholz1998, Rice1983,  JRiceetal_83,  Andrews1985}.  External boundaries of the domain are closed with a general linear energy-stable boundary conditions, modeling various geophysical phenomena.   Further, we design a numerical flux obeying the eigen-structure of the PDE and the underlying physics at the internal and external DG element boundaries.

 The paper  begins the development of  a unified provably stable and robust adaptive DG framework for the numerical treatment of  1) nonlinear frictional sliding in elastic solids,  2) for coupling classical DG inter-element interfaces in elastic solids where slip is not permitted, and 3) numerical enforcement of external well-posed boundary conditions modeling various geophysical phenomena. This is critical for reliable and efficient numerical simulations of dynamic earthquake ruptures and time-domain propagating elastic waves in complex Earth models, and numerical simulations of engineering applications where frictional failure can be fatal. 
  We  remark that an analogous  method has been used in a finite difference framework \cite{DuruandDunham2016} to model frictional sliding during dynamic earthquake ruptures \cite{Harris2018, Scholz1998, Rice1983, JRiceetal_83}. However, static and/or dynamic adaptive mesh refinement in a finite difference setting is a great challenge. More importantly, this is the first time physical conditions, such as friction, have been proposed to be used to couple locally adjacent DG elements together, to the global domain.
 For clarity, we will focus on a one space dimensional (1D) model problem. We remark that most of the difficulties we hope to alleviate often appear in higher (2D and 3D) space dimensions.  However, the 1D model problem is simple and sufficient to  demonstrate the fundamentals of our idea, and the procedure and analysis can be easily extended to the multi-dimensional linear elastic wave equation in complex geometries. 
  
 We note that the elastic wave equation is hyperbolic,  can be decomposed into characteristics, and the characteristics are the natural carrier of information in the system.  
 The holy grail of prescribing well-posed boundary conditions is to ensure that boundary data preserve the amplitude of the outgoing characteristics on the boundary. Boundary conditions can then be enforce by modifying the amplitude of the incoming characteristics \cite{GustafssonKreissOliger1995}.
 In order to generate boundary/interface data, we solve  a Riemann-like problem and constrain the solution so that the amplitude of the outgoing characteristic is preserved and the solution satisfies physical boundary/interface conditions (eg. force balance and friction law). The solution is exact and unique.
To communicate data across internal and external element boundaries, we penalize the numerical boundary/interface data on the boundary/interface against incoming characteristics only.  
Next we construct  a flux fluctuation vector obeying the structure of the underlying PDE. 
Finally, we append the flux fluctuation vector to the discretized PDE with physically motivated penalty weights.   By construction our choice of penalty parameters yield an upwind scheme and a discrete energy estimate analogous to the continuous energy estimate. We present numerical experiments, using Lagrange basis with Gauss-Legendre-Lobatto (GLL) quadrature nodes and Gauss-Legendre (GL) quadrature nodes, separately,    verifying accuracy and numerical stability. We present 2D numerical experiments demonstrating the extension of our method to multiple spatial dimensions, verifying high order accuracy  for Rayleigh surface waves and make comparisons with the Rusanov flux. We simulate dynamic earthquake rupture model problems in 1D and 2D, demonstrating the robustness of the method.

The remainder of the paper will proceed as follows. In section 2 we present a model problem and derive continuous energy estimates that our numerical approximation should emulate. Boundary and interface data are constructed in section 3. In section 4, we present the DG method and the new boundary and inter-element procedures, beginning from the integral formulation down to numerical approximations. Numerical stability is proven in section 5, using the energy method. In section 6, we present some numerical examples. In section 7, we draw conclusions and suggest future work.
\section{Model problem}\label{sec:model_problem}
Consider the elastic wave equation in a heterogeneous one space dimensional domain
{\small
\begin{equation}\label{eq:elastic_1D}
\begin{split}
\rho(x) \frac{\partial v}{\partial t} &= \frac{\partial \sigma}{\partial x}, \quad \frac{1}{\mu (x)} \frac{\partial \sigma}{\partial t} = \frac{\partial v}{\partial x}, \quad x \in [0, L], \quad t \ge 0.
\end{split}
\end{equation}
}
The unknowns are $v(x,t)$,  the particle velocity, and $\sigma(x,t)$, the stress field. The material parameter $ \rho(x)$ is the mass density and  $\mu(x)$ is the shear modulus. Define the shear wave-speed by $c_s = \sqrt{\mu/\rho}$. In order to complete the statement of the problem,  and define a well-posed an initial boundary value problem (IBVP), we will need initial conditions at $t = 0$ and boundary conditions at $x = 0, L$.  We prescribe the initial condition in $\mathbb{L}^2(0, L)$,
{
\begin{align}\label{eq:initial_data}
(v(x, 0), \sigma(x, 0)) =  (v_0(x), \sigma_0(x)) \in \mathbb{L}^2(0, L).
\end{align}
}
Now we introduce the shear impedance $Z_s$, the left-going characteristic $p$,  and the right-going characteristic $ q$ defined by
{\small
\begin{equation}\label{eq:characteristics}
p = \frac{1}{2}\left(Z_sv + \sigma\right), \quad q = \frac{1}{2}\left(Z_s v - \sigma\right), \quad Z_s = \rho c_s.
\end{equation}
}
Note  that at the left boundary, $x = 0$, ${p}$ is the outgoing characteristic and $q$ is the incoming characteristic.
Conversely, at  the right boundary, $x = L$,  ${q}$ is the outgoing characteristic and $p$ is the incoming characteristic.

\subsection{Boundary conditions}\label{subsec:boundary_conditions}
When prescribing well-posed boundary conditions, one thing we earnestly seek  is to ensure that boundary data preserve the amplitude of the outgoing characteristics on the boundary. Boundary conditions can then be enforced by modifying the amplitude of the incoming characteristics. In general,  boundary data for the incoming characteristics  can be expressed as a linear combination of the outgoing characteristics \cite{GustafssonKreissOliger1995}.  We consider the  general linear well-posed boundary conditions 
{\small
\begin{equation}\label{eq:boundary_conditions}
q = r_0 p, \quad \text{at} \quad x = 0, \quad \text{and} \quad p = r_L q, \quad \text{at} \quad  x = L,
\end{equation}
}
with the reflection coefficients $r_0$, $r_L$ being real numbers and  $|r_0|, |r_L| \le 1$. 
The amplitude of the incoming characteristic is altered via the reflection coefficients $ r_0$, $r_L $. Note that at $x = 0$,  while  $r_0 = -1$ yields a clamped wall, $r_0 = 0$  yields  an absorbing boundary, and  with $r_0 = 1$  we have a free-surface boundary condition. Similarly, at $x = L$, $r_L = -1$ yields a clamped wall, $r_L = 0$ yields an absorbing boundary, and  $r_L = 1$  gives a free-surface boundary condition.
We have tacitly considered homogeneous boundary forcing, however, the analysis carries over  to the case of inhomogeneous boundary forcing. By rearranging and collecting terms together, the boundary condition \eqref{eq:boundary_conditions} can be rewritten in terms of the primitive variables, $v, \sigma$,  having
{\small
\begin{equation}\label{eq:BC}
\begin{split}
B_0(v, \sigma, Z_{s}, r_0): =\frac{Z_{s}}{2}\left({1-r_0}\right){v} -\frac{1+r_0}{2} {\sigma} = 0,  \quad \text{at} \quad x = 0, \\
 B_L(v, \sigma, Z_{s}, r_L): =\frac{Z_{s}}{2} \left({1-r_L}\right){v} + \frac{1+r_L}{2}{\sigma} = 0, \quad \text{at} \quad  x = L.
 \end{split}
\end{equation}
}

To see that the IBVP, \eqref{eq:elastic_1D} with \eqref{eq:boundary_conditions} or \eqref{eq:BC},  is well-posed we seek an integral form of the PDE \eqref{eq:elastic_1D} by multiplying the elastic wave equation by a set of arbitrary  test functions $(\phi_v(x), \phi_\sigma(x)) \in \mathbb{L}^2(0, L)$ and integrate over the whole domain. We have
{\small
\begin{equation}\label{eq:weak_form_1}
\begin{split}
\int_0^L\left({\rho(x)} \phi_v(x)\frac{\partial v(x, t)}{\partial t}  - \phi_v(x)\frac{\partial \sigma(x, t)}{\partial x} \right) dx = 0,
\end{split}
\end{equation}
\begin{equation}\label{eq:weak_form_2}
\begin{split}
\int_0^L\left(\frac{1}{\mu(x)}\phi_{\sigma}(x)\frac{\partial \sigma(x, t)}{\partial t}  - \phi_{\sigma}(x)\frac{\partial v(x, t)}{\partial x} \right) dx = 0.
\end{split}
\end{equation}
}
We introduce the mechanical energy defined by
{\small
\begin{equation}\label{eq:continuous_energy}
E(t) = \frac{1}{2}\int_0^L{\left({\rho(x)} v^2(x, t) + \frac{1}{\mu(x)}\sigma^2(x, t)\right) dx},
\end{equation}
}
where $E(t)$ is the sum of the kinetic energy and the strain energy.

Now, replace $\phi_v(x)$ with $v(x,t)$ in \eqref{eq:weak_form_1} and $\phi_{\sigma}(x)$ with ${\sigma}(x,t)$ in \eqref{eq:weak_form_2}. Integrating the second term in \eqref{eq:weak_form_1} by parts, and summing the equations \eqref{eq:weak_form_1}--\eqref{eq:weak_form_2}, we find that the spatial derivatives vanish. We have 
{\small
\begin{equation}\label{eq:continuous_energy_rate}
\frac{d E(t)}{dt} = -v(0, t)\sigma(0, t) + v(L, t)\sigma(L, t).
\end{equation}
}
From the boundary conditions \eqref{eq:BC}, it is easy to check that  $v(0, t)\sigma(0, t) \ge 0$ and $v(L, t)\sigma(L, t) \le 0$, for all $|r_0|, |r_L| \le 1$. 
The boundary terms in \eqref{eq:continuous_energy_rate} are negative semi-definite, $-v(0, t)\sigma(0, t) + v(L, t)\sigma(L, t) \le 0$, and dissipative. This energy loss through the boundaries is what the numerical method should mimic. Since boundary terms are negative semi-definite,   we therefore have
{\small
\begin{equation}\label{eq:continuous_energy_rate_1}
\frac{d E(t)}{dt}  \le 0. 
\end{equation}
}
Thus, the  mechanical energy is bounded by the initial mechanical energy for all times, $E(t)  \le E(0) $.
\subsection{Interface conditions}\label{subsec:interface_conditions}
In this section we define physical interface  conditions that must be satisfied when elastic blocks are in contact. One idea of this study is to use friction to couple DG elements to the global domain. Therefore, we consider a generic nonlinear  friction law, accommodating frictional slip motion.

To begin, consider the domain $\Omega = \Omega_{-}\cup  \Omega_{+}$, with $  \Omega_{-}:= [0, x_0]$,  $\Omega_{+}:= [x_0, L]$, $0<x_0<  L$.  We denote field variables and material parameters in the sub-domains  $  \Omega_{\pm}$ with the superscripts $\pm$:  $v^{\pm}$,  $\sigma^{\pm}$, $\rho^{\pm}$,  $\mu^{\pm}$, $Z_s^{\pm}$. Since there are two characteristics going in and out of the interface we need exactly two interface conditions coupling the elastic subdomains. Define tractions $T^- = \sigma^-$, $T^{+} = - \sigma^{+}$, acting on the interface. We begin with force balance:
\begin{align}
T^- = -T^+ \iff \sigma^{-} = \sigma^{+} = \sigma.
\end{align}
To complete the interface condition we introduce discontinuity in particle velocity:  $\llbracket v \rrbracket := v^{+} - v^{-}$, and  define the absolute slip-rate $V: = \left|\llbracket v \rrbracket\right|$. 
We introduce the compressive normal stress $\sigma_n > 0$ and define the frictional constitutive relation, we have
\begin{align}
  \sigma =  \alpha  \llbracket v \rrbracket, \quad \alpha = \sigma_n\frac{f(V)}{V} \ge 0 .
\end{align}
Here $f(V) \ge 0$ with $f(0) =0$ is the nonlinear friction coefficient.
Note that
\begin{align}
 {V} \to 0 \iff  \alpha = \sigma_n\frac{f(V)}{V} \to \infty .
\end{align}
For later use, we summarize the interface condition:
{
\begin{align}\label{eq:physical_interface}
\text{force balance}:  \quad &\sigma^{-} = \sigma^{+} = \sigma, \nonumber \\
\text{friction law}: \quad & \sigma =\alpha \llbracket v \rrbracket, \quad \alpha = \sigma_n\frac{f(V)}{V}\ge 0.
\end{align}
}
Tractions on the interface are related to particle velocities via  $\sigma = \alpha \llbracket v \rrbracket$, with $ \alpha \ge 0$.
The  parameter {\small $\alpha \ge 0$} is related to the nonlinear frictional strength of the interface. Note that there are two limiting values, a locked interface: {\small \small $ \alpha \to \infty \iff \lJump  {v}  \rJump \to 0$}, and a frictionless interface: {\small \small $\alpha \to 0 \iff \sigma \to 0$}. These limiting cases are degenerate but physically feasible.

Since  {\small \small $ \alpha \to \infty \iff \lJump  {v}  \rJump \to 0$}, the limit $ \alpha \to \infty$ in  \eqref{eq:physical_interface} is an alternative way of expressing the continuity of particle velocities across an interface, thus gives the natural condition to be used to patch DG elements together, when slip motion is not present. However, we can model nonlinear frictional slip motion by replacing $f(V)$ in \eqref{eq:physical_interface} with an appropriate friction law \cite{Scholz1998, Rice1983, JRiceetal_83}.

We define the mechanical energy in each subdomain by
{\small
\begin{equation}\label{eq:continuous_energy_element}
E^{\pm}(t) = \frac{1}{2}\int_{\Omega_{\pm}}{\left({\rho^{\pm}(x)} |v^{\pm}(x, t)|^2 + \frac{1}{\mu^{\pm}(x)}|\sigma^{\pm}(x, t)|^2\right) dx}.
\end{equation}
}
The elastic wave equation with the physical  interface condition \eqref{eq:physical_interface},   satisfies the energy equation
{\small
\begin{equation}\label{eq:continuous_energy_rate_1}
\frac{d E(t)}{dt}  = -\sigma \llbracket v \rrbracket -v^-(0, t)\sigma^-(0, t) + v^+(L, t)\sigma^+(L, t),
\end{equation}
}
with $E(t) = E^{-}(t) + E^{+}(t)$. The interior term $ -\sigma \llbracket v \rrbracket$ is the rate of work done by friction during frictional slip,  which is dissipated as heat. Note the negative work rate,  and  since for $\alpha \ge 0$ we have $\sigma \llbracket v \rrbracket = \alpha \llbracket v \rrbracket ^2 = \frac{1}{\alpha} \sigma ^2 \ge 0$.  At the limit  {\small \small $ \alpha \to \infty \iff \lJump  {v}  \rJump \to 0$} or {\small \small $\alpha \to 0 \iff \sigma \to 0$},  the interior term vanishes, {\small ${\sigma} {\lJump {v}  \rJump}  \to 0$}.  Thus, at {\small $\alpha \to \infty$ or $\alpha \to 0$}, the energy equation \eqref{eq:continuous_energy_rate_1} is completely equivalent to \eqref{eq:continuous_energy_rate}.

Our main objective is to formulate an inter-element procedure incorporating the physical interface condition \eqref{eq:physical_interface} and the boundary condition \eqref{eq:BC}, so that a discrete energy equation analogous to \eqref{eq:continuous_energy_rate_1} can be derived. The procedure should be formulated in a unified manner such that numerical flux functions are compatible with the general linear boundary condition \eqref{eq:boundary_conditions} or \eqref{eq:BC}.  Furthermore,  the procedure should be efficient for explicit time stepping schemes, thus   avoiding numerical stiffness, for all  {\small $0\le \alpha \le \infty$}. The numerical treatment should be easily extended to higher space dimensions (2D and 3D). 

\section{ Hat-variables}\label{sec:hat_variables}
We will now reformulate the boundary condition  \eqref{eq:boundary_conditions} and  interface condition \eqref{eq:physical_interface} by introducing transformed (hat-) variables so that we can simultaneously construct (numerical) boundary/interface data for particle velocities and tractions.  The hat-variables encode the solution of the IBVP on the boundary/interface. The hat-variables  will be constructed such that they preserve the amplitude of the outgoing characteristics and  satisfy the physical boundary conditions \cite{DuruandDunham2016} exactly. To be more specific, the hat-variables are solutions of the Riemann problem constrained against physical boundary/interface conditions \eqref{eq:BC}  and \eqref{eq:physical_interface}.
\subsection{Boundary data}\label{subsec:boundary_hat_variables}
We will  construct boundary data which satisfy the physical boundary conditions \eqref{eq:BC} exactly and preserve the amplitude of the outgoing characteristic  $p$ \text{at}   $x = 0$, and  $q$ at $x = L$.
To begin, define the hat-variables preserving the amplitude of outgoing characteristics
{\small
\begin{equation}\label{eq:BC_hat_1bc}
\frac{1}{2}\left(Z_{s}(0)\widehat{v}_0 + \widehat{\sigma}_0\right) = p_0, \quad \frac{1}{2}\left({Z_{s}(L)}\widehat{v}_L - \widehat{\sigma}_L \right)= q_L,
\end{equation}
}
with
{\small
\begin{equation}\label{eq:characteristics_2}
p_0 = \frac{1}{2}\left(Z_s(0)v(0,t) + \sigma(0,t)\right), \quad q_L = \frac{1}{2}\left(Z_s(L) v(L,t) - \sigma(L,t)\right).
\end{equation}
}
Since hat-variables also satisfy the physical boundary condition, we must have
{\small
\begin{equation}\label{eq:BC_hat_2bc}
\frac{Z_{s}(0)}{2}\left({1-r_0}\right)\widehat{v}_0 -\frac{1+r_0}{2} \widehat{\sigma}_0 = 0, \quad \frac{Z_{s}(L)}{2} \left({1-r_L}\right)\widehat{v} _L+ \frac{1+r_L}{2}\widehat{\sigma}_L = 0.
\end{equation}
}
The algebraic problem for the hat-variables, defined by equations \eqref{eq:BC_hat_1bc} and \eqref{eq:BC_hat_2bc},  has a unique solution, namely
{\small
\begin{align}\label{eq:data_hat}
\widehat{v}_0  = \frac{(1+r_0)}{Z_{s}(0)}p_0, \quad \widehat{\sigma}_0  = {(1-r_0)}p_0, \nonumber \\
\widehat{v}_L = \frac{(1+r_L)}{Z_{s}(L)}q_L,  \quad \widehat{\sigma}_L  = {-(1-r_L)}q_L.
\end{align}
}
The expressions in \eqref{eq:data_hat} define a rule to update particle velocities and tractions on the external boundaries $x = 0, L$,
{\small
\begin{align}\label{eq:boundary_data_hat}
v(x, t)  &= \widehat{v}_0(x, t) , \quad {\sigma}(x, t) = \widehat{\sigma}_0(x, t) , \quad \text{at} \quad x = 0, \nonumber \\
v(x, t)  &= \widehat{v}_L(x, t)  , \quad {\sigma}(x, t) = \widehat{\sigma}_L(x, t) , \quad \text{at} \quad x = L.
\end{align}
}
It is particularly important to note that the boundary procedure \eqref{eq:boundary_data_hat} is equivalent  to the original boundary condition \eqref{eq:boundary_conditions}.
To verify this, consider a free-surface boundary condition at $x = 0$, with $r_0 = 1$. From \eqref{eq:data_hat}  and \eqref{eq:boundary_data_hat} we have 
$
{\sigma}(0, t) = \widehat{\sigma}_0(0, t) = 0, 
$
and 
$
{v}(0, t) = \widehat{v}_0(0, t) = v(0,t).
$
The traction on the boundary, at $x = 0$, vanishes and the particle velocity on the boundary, at $x = 0$, is not altered by the boundary procedure \eqref{eq:boundary_data_hat}.

By construction, the hat-variables $\widehat{v}_0, \widehat{\sigma}_0$, $\widehat{v}_L, \widehat{\sigma}_L$ satisfy the following algebraic identities:
\begin{subequations}\label{eq:identity_bc}
\small
\begin{equation}\label{eq:identity_1_bc}
\widehat{p}_0 = p_0,  \quad \widehat{q}_L = q_L,
\end{equation}
\begin{equation}\label{eq:identity_2_bc}
\left(p_0\right)^2-\left(\widehat{q}_0\right)^2 = Z_s(0)\widehat{\sigma}_0\widehat{v}_0, \quad \left(q_L\right)^2 -\left(\widehat{p}_L\right)^2 = -Z_s(L)\widehat{\sigma}_L\widehat{v}_L,
\end{equation}
\begin{equation}\label{eq:identity_3_bc}
 \widehat{\sigma}_0\widehat{v}_0 = \frac{1-r_0^2}{Z_{s}(0)}|p_0|^2 \ge 0, \quad \widehat{\sigma}_L\widehat{v}_L = -\frac{1-r_L^2}{Z_{s}(L)}|q_0|^2 \le 0.
\end{equation}
\end{subequations}
The first identity \eqref{eq:identity_1_bc} holds by  definition \eqref{eq:BC_hat_1bc}. Using \eqref{eq:identity_1_bc} in $\left(p_0\right)^2-\left(\widehat{q}_0\right)^2$ and  $ \left(q_L\right)^2 -\left(\widehat{p}_L\right)^2$ gives the second identity  \eqref{eq:identity_2_bc}. From the solutions of the hat-variables in \eqref{eq:data_hat} it is clear that \eqref{eq:identity_3_bc} holds.
The algebraic identities \eqref{eq:identity_1_bc}--\eqref{eq:identity_3_bc} will be crucial in proving numerical stability.

\subsection{Interface data}\label{subsec:interface_hat_variables}
Similarly, for the interface we define the outgoing characteristics 
\begin{equation}\label{eq:outgoing_charact}
q^{-}:= \frac{1}{2}\left(Z_{s}^{-} v^{-} - \sigma^{-}\right), \quad p^{+}:= \frac{1}{2}\left(Z_{s}^{+} v^{+} + \sigma^{+}\right),
\end{equation}
that must be preserved by the interface data. By combining \eqref{eq:outgoing_charact} with force balance, $\sigma^{-} = \sigma^{+} = \sigma$, we obtain
\begin{equation}\label{eq:radiation_damping_line}
{\sigma} = \Phi  - \eta  \llbracket {v} \rrbracket,
\end{equation}
where
\[
\Phi = \eta \left(\frac{2}{Z_{s}^{+}}p^{+} - \frac{2}{Z_{s}^{-}} q^{-}\right), \quad  \eta = \frac{Z_{s}^{-}Z_{s}^{+}}{Z_{s}^{+}+Z_{s}^{-}} > 0.
\]
Note that $\Phi$ is the stress transfer functional and $\eta  \llbracket {v} \rrbracket$ is the radiation damping term \cite{DuruandDunham2016, GeubelleRice1995}. Equation \eqref{eq:radiation_damping_line} arises naturally in the boundary integral formulation of linear elasticity \cite{GeubelleRice1995}. In particular, ${\sigma} = \Phi$ is the traction on a locked  interface, $\llbracket {v} \rrbracket = 0$, which is altered by outgoing wave radiation, according to \eqref{eq:radiation_damping_line}, when the interface is slipping, $\llbracket {v} \rrbracket \ne 0$.

We want to construct interface data $\widehat{v}^{-}, \widehat{\sigma}^{-}$, $\widehat{v}^{+}, \widehat{\sigma}^{+}$, and the absolute slip-rate $\widehat{V} = |{\lJump  \widehat{v}  \rJump }| \ge 0$, such that the data satisfy the physical interface conditions (force balance + friction law) 
\begin{align}\label{eq:physical_interface_hat}
\text{force balance}:  \quad &{ \widehat{\sigma}^{-} = \widehat{\sigma}^{+} = \widehat{\sigma}}, \nonumber\\
\text{friction law}: \quad &  \widehat{\sigma} = \alpha {\lJump  \widehat{v}  \rJump },  \quad \alpha = \sigma_n\frac{f(\widehat{V})}{\widehat{V}}\ge 0,
\end{align}
 and preserve the amplitude of the outgoing characteristics
\begin{align}\label{eq:charac_interface_hat}
\widehat{q}^{-}:=\frac{1}{2}\left(Z_{s}^{-} \widehat{v}^{-} - \widehat{\sigma}^{-}\right) = q^{-}, \quad \widehat{p}^{+}:= \frac{1}{2}\left(Z_{s}^{+} \widehat{v}^{+} + \widehat{\sigma}^{+}\right) = p^{+}.
\end{align}
As before, combining both equations in \eqref{eq:charac_interface_hat}  and enforcing force balance,  ${ \widehat{\sigma}^{-} = \widehat{\sigma}^{+} = \widehat{\sigma}}$, defined in \eqref{eq:physical_interface_hat},  we obtain
\[
\widehat{\sigma} = \Phi  - \eta  \llbracket \widehat{v} \rrbracket.
\]
Thus, we obtain the nonlinear algebraic problem for tractions and slip-rate,
\begin{equation}\label{eq:algebraic_problem}
\widehat{\sigma} = \Phi  - \eta  \llbracket \widehat{v} \rrbracket, \quad   \widehat{\sigma} = \alpha \llbracket \widehat{v} \rrbracket,   \quad \alpha = \sigma_n\frac{f(\widehat{V})}{\widehat{V}}\ge 0.
\end{equation}
However, if the friction coefficient $f(\widehat{V})$ is linear the corresponding algebraic problems in \eqref{eq:algebraic_problem} will be linear.
By combing the two equations in \eqref{eq:algebraic_problem} to
\begin{equation}\label{eq:algebraic_problem_0}
\sigma_n{f(\widehat{V})}  + \eta \widehat{V} = |\Phi|,
\end{equation}
which is a nonlinear algebraic equation for the absolute slip-rate $\widehat{V}\ge 0$. We can now solve  \eqref{eq:algebraic_problem_0} for the absolute slip-rate $\widehat{V}$ using any root finding algorithm, and compute $\alpha \ge 0$.
 The above algebraic problem \eqref{eq:algebraic_problem} has a unique solution which is solved exactly,
\begin{equation}\label{eq:sol_linear_equation}
\widehat{\sigma} = \frac{\alpha}{\eta + \alpha}\Phi , \quad  \llbracket \widehat{v} \rrbracket = \frac{1}{\eta + \alpha}\Phi,   \quad \alpha = \sigma_n\frac{f(\widehat{V})}{\widehat{V}}\ge 0.
\end{equation}
We therefore have
\[
\widehat{\sigma}^{-}  = \widehat{\sigma}^{+} = \widehat{\sigma},
\]
and
\[
\widehat{v}^{-}  = \frac{1}{Z_{s}^{+} }\left(2p^{+} - \widehat{\sigma}^{+}\right) - \llbracket \widehat{v} \rrbracket, \quad
\widehat{v}^{+}  = \frac{1}{Z_{s}^{-} }\left(2q^{-} + \widehat{\sigma}^{-}\right) + \llbracket \widehat{v} \rrbracket.
\]
We have constructed a rule to update tractions and particle velocities  on the interface, $x = x_0$,
\begin{align}\label{eq:interface_data_hat}
{\sigma}^{-} &= \widehat{\sigma}^{-}, \quad  {\sigma}^{+}  =  \widehat{\sigma}^{+},\nonumber \\
{v}^{-} &= \widehat{v}^{-},  \quad {v}^{+}  = \widehat{v}^{+}.
\end{align}
In \eqref{eq:interface_data_hat}, we have equivalently redefined the physical interface condition \eqref{eq:physical_interface}.

By construction, the hat-variables $\widehat{v}^{-}, \widehat{\sigma}^{-}$, $\widehat{v}^{+}, \widehat{\sigma}^{+}$ satisfy the following algebraic identities:
\begin{subequations}\label{eq:identity}
\small
\begin{equation}\label{eq:identity_1}
\widehat{p}^+ = p^+,  \quad \widehat{q}^- = q^-,
\end{equation}
\begin{equation}\label{eq:identity_2}
\left(p^+\right)^2-\left(\widehat{q}^+\right)^2 = Z_s^+\widehat{\sigma}\widehat{v}^+, \quad \left(q^-\right)^2 -\left(\widehat{p}^-\right)^2 = -Z_s^-\widehat{\sigma}\widehat{v}^-,
\end{equation}
\begin{equation}\label{eq:identity_3}
\frac{1}{Z_s^+}\left(\left(p^+\right)^2-\left(\widehat{q}^+\right)^2\right)  + \frac{1}{Z_s^-}\left(\left(q^-\right)^2 -\left(\widehat{p}^-\right)^2\right) =  \widehat{\sigma} \llbracket \widehat{v} \rrbracket = \frac{\alpha}{(\eta + \alpha)^2}|\Phi|^2,
\end{equation}
\end{subequations}
where
\begin{align*}\label{eq:charac_interface_hat}
\widehat{p}^{-}:=\frac{1}{2}\left(Z_{s}^{-} \widehat{v}^{-} +\widehat{\sigma}^{-}\right) , \quad \widehat{q}^{+}:= \frac{1}{2}\left(Z_{s}^{+} \widehat{v}^{+} - \widehat{\sigma}^{+}\right) .
\end{align*}
The first identity \eqref{eq:identity_1} holds by the definition \eqref{eq:charac_interface_hat}. Using \eqref{eq:identity_1} in $\left(p^+\right)^2-\left(\widehat{q}^+\right)^2$ and  $ \left(q^-\right)^2 -\left(\widehat{p}^-\right)^2$ gives the second identity  \eqref{eq:identity_2}. The third identity \eqref{eq:identity_3} follows trivially from \eqref{eq:identity_2} with $\widehat{\sigma} = \frac{\alpha}{\eta + \alpha}\Phi$, $\widehat{v}^+ - \widehat{v}^- := \llbracket \widehat{v} \rrbracket = \frac{1}{\eta + \alpha}\Phi$.
The data is unique and exact. Note the consistency at the limits: {\small \small $ \alpha \to \infty \iff \lJump  \widehat{v}  \rJump \to 0$, $\widehat{\sigma}\lJump  \widehat{v}  \rJump \to 0$},  and {\small \small $\alpha \to 0 \iff \widehat{\sigma} \to 0$, $\widehat{\sigma} \lJump  \widehat{v}  \rJump \to 0$.
As before, the identities defined in \eqref{eq:identity_1}--\eqref{eq:identity_3} will be crucial in proving numerical stability.
\section{ The discontinuous Galerkin method}
We begin by discretizing the interval $x \in [0, L]$ into $K$ elements denoting the $k$-th element by $e^k = [x_k, x_{k+1}]$, where $k = 1, 2, \dots, K$, with $x_1 = 0$ and $x_{K+1} = L$.
Therefore, the integral form \eqref{eq:weak_form_1}--\eqref{eq:weak_form_2} yield
\begin{equation}\label{eq:velocity_weak}
\begin{split}
\sum_{k = 1}^{K}\int_{x_k}^{x_{k+1}}\left({\rho(x)} \phi_v(x)\frac{\partial v(x, t)}{\partial t}  - \phi_v(x)\frac{\partial \sigma(x, t)}{\partial x} \right) dx = 0, 
\end{split}
\end{equation}
\begin{equation}\label{eq:stress_weak}
\small
\sum_{k = 1}^{K}\int_{x_k}^{x_{k+1}}\left(\frac{1}{\mu(x)}\phi_{\sigma}(x)\frac{\partial \sigma(x, t)}{\partial t}  - \phi_{\sigma}(x)\frac{\partial v(x, t)}{\partial x} \right) dx = 0.
\end{equation}
\subsection{Inter-element and boundary procedure, and energy identity}
We will begin the development and construction  of the inter-element and boundary procedure for the continuous integral form \eqref{eq:velocity_weak}--\eqref{eq:stress_weak}. As we will see later the procedure and  analysis will naturally carry over when numerical approximations are introduced. We will end the discussion with the derivation of an energy equation analogous to \eqref{eq:continuous_energy_rate}.

Next   we consider the element boundaries, $x = x_k, x_{k+1}$, and generate boundary and interface data $\widehat{v}(x, t)$, $ \widehat{\sigma} (x, t)$.
Note that, by both physical and mathematical considerations, the only way information can be propagated into an element is through the incoming characteristics on the boundaries,  $q$ at $x_k$ and $p$ at $x_{k+1}$. We construct flux fluctuations by penalizing data against incoming characteristics $p$ and $q$,
\begin{equation}
\small
 F(x_k, t):=q-\widehat{q} =\frac{Z_{s}(x_k)}{2} \left(v(x_k, t)-\widehat{v}(x_k, t) \right) - \frac{1}{2}\left(\sigma(x_k, t)- \widehat{\sigma} (x_k, t)\right),
\end{equation}
\begin{equation}
\small
G(x_{k+1}, t) := p-\widehat{p} = \frac{Z_{s}(x_{k+1}) }{2} \left(v(x_{k+1}, t) -\widehat{v}(x_{k+1}, t)  \right) + \frac{1}{2}\left(\sigma(x_{k+1}, t) - \widehat{\sigma}(x_{k+1}, t) \right).
\end{equation}
Note that $q$ is the incoming characteristic at the left element boundary $x = x_k$ and $p$ is incoming characteristic at right element boundary $x = x_{k+1}$.
Therefore, $F(x_k, t)$ penalizes data against the incoming characteristic at $x = x_k$ and $G(x_{k+1}, t)$ penalizes data against the incoming characteristic at $x = x_{k+1}$. 
\begin{remark}
Note the uniform treatment of all DG element boundaries $x = x_k, x_{k+1}$, by the flux fluctuations $F(x_{k}, t)$ and $G(x_{k+1}, t)$.
The difference between external element  boundaries $x_k = 0$, $x_{k+1} =L$ and internal element  boundaries $x_k > 0$, $x_{k+1} < L$ is determined by the algebraic problem yielding the corresponding hat-variables $\widehat{v}$,  $\widehat{\sigma}$.
\end{remark}
Since we have not introduced any approximation yet, we must have $v(x_k, t) \equiv \widehat{v}(x_k, t)$,  $\sigma(x_k, t) \equiv  \widehat{\sigma}(x_k, t)$ and $v(x_{k+1}, t) \equiv  \widehat{v}(x_{k+1}, t)$,  $\sigma(x_{k+1}, t) \equiv  \widehat{\sigma}(x_{k+1}, t)$.    Thus, at the external boundaries, at $x_1 = 0$, $x_{K+1} = L$,  the fluctuations satisfy  the boundary operator $B_0(v(x_{1}, t), \sigma(x_{1}, t), Z_{s}(x_1), r_0)=0$, $B_L(v(x_{K+1}, t), \sigma(x_{K+1}, t), Z_{s}(x_{K+1}), r_L)=0$, obtaining
\begin{align}
 F(x_1, t)\equiv B_0(v(x_{1}, t), \sigma(x_{1}, t), Z_{s}(x_1), r_0)=0, \quad  G(x_{K+1}, t)\equiv B_L(v(x_{K+1}, t), \sigma(x_{K+1}, t), Z_{s}(x_{K+1}), r_L)=0.
\end{align}
Next, append the flux fluctuations, $F(x_{k}, t) \to 0 $, $G(x_{k+1}, t) \to 0$,   to the integral form \eqref{eq:velocity_weak}--\eqref{eq:stress_weak} with special penalty weights. Thus, we have the weak form
\begin{equation}\label{eq:velocity_weak_pen}
\begin{split}
&\sum_{k = 1}^{K}{\left(\int_{x_k}^{x_{k+1}}\left({\rho(x)} \phi_v(x)\frac{\partial v(x, t)}{\partial t}  - \phi_v(x)\frac{\partial \sigma(x, t)}{\partial x} \right) dx \right)}+ \sum_{k = 1}^{K}{\left(\phi_v(x_k) F(x_k, t) +  \phi_v(x_{k+1})G(x_{k+1}, t)\right)} = 0, 
\end{split}
\end{equation}
\begin{equation}\label{eq:stress_weak_pen}
\begin{split}
&\sum_{k = 1}^{K}\left(\int_{x_k}^{x_{k+1}}\left(\frac{1}{\mu(x)}\phi_{\sigma}(x)\frac{\partial \sigma(x, t)}{\partial t}  - \phi_{\sigma}(x)\frac{\partial v(x, t)}{\partial x} \right) dx \right) -\sum_{k = 1}^{K}\left( \frac{\phi_{\sigma}(x_{k})}{Z_s(x_{k})}F(x_{k}, t) -   \frac{\phi_{\sigma}(x_{k+1})}{Z_s(x_{k+1})}G(x_{k+1}, t)\right) = 0.
\end{split}
\end{equation}

We have  weakly implemented the boundary and interface conditions by penalizing data against the incoming characteristics at the element boundaries at $x = x_k$ and $x = x_{k+1}$.  Recall that we are yet to introduce numerical approximations, therefore the flux fluctuations vanish identically, that is  $G(x_{k+1}, t) = F(x_{k}, t)  = 0$. However, when numerical approximations are introduced the flux fluctuations will be proportional to the truncation error.  Note  that the external physical boundary conditions and the inter-element conditions are treated in a unified manner.

The penalty weights have been chosen such that the physical dimensions of all terms in equations \eqref{eq:velocity_weak_pen}--\eqref{eq:stress_weak_pen} match. For instance in the stress equation \eqref{eq:stress_weak_pen}, we have penalized the flux functions by the shear admittance, $1/Z_s(x)$. This is motivated by a dimensional analysis. As we will see later, this physically motivated penalty weight is also critical for numerical stability.

\begin{remark}
The following remarks are of significant importance, and summarize the procedure:
\begin{itemize}
\item[1.] All DG inter-element faces are held together by a frictional strength, $\alpha \ge 0$.
\item[2.] Classical DG element internal faces where slip is not permitted have infinite frictional strength, $\alpha \to \infty$, and can never slip.
\item[3.] Weak  interfaces have  finite frictional strength,  $\alpha \ge 0$, and the  slip  motion is governed by a friction law.
\item[4.] External DG element faces, at $x = 0, L$, are closed with the linear well-posed boundary conditions \eqref{eq:boundary_conditions}.
\item[5.]  We construct transformed (hat-) variables that encode the solutions of the IBVP at element faces.
\item[6.] By construction the DG flux fluctuations, $G(x_{k+1}, t)$  $F(x_{k}, t)$, have been designed to satisfy the boundary condition  \eqref{eq:boundary_conditions} and  the frictional interface condition \eqref{eq:physical_interface} exactly.
\end{itemize}
\end{remark}

We can now state our first main result.
\begin{theorem}\label{theorem:main_1}
The weak form \eqref{eq:velocity_weak_pen}--\eqref{eq:stress_weak_pen} satisfies the energy identity
\begin{equation}\label{eq:weak_energy}
\begin{split}
\frac{d}{dt} {E}(t) = &-\sum_{k=1}^{K}\left(\frac{1}{Z_{s}(x_k)}|F(x_k, t)|^2 + \frac{1}{Z_{s}(x_{k+1})}|G(x_{k+1}, t)|^2 \right) - \sum_{k=2}^{K}\frac{\alpha(x_k)}{\left(\eta(x_k) + \alpha(x_k)\right)^2}|\Phi(x_k)|^2  \\
& - \frac{1-r_0^2}{Z_{s}(0)}|p_0|^2 -\frac{1-r_L^2}{Z_{s}(L)}|q_L |^2,
\end{split}
\end{equation}
with $p_0$, $q_L$ defined in \eqref{eq:characteristics_2}.
\end{theorem}
\begin{proof}
As in section \eqref{subsec:boundary_conditions}, by replacing  $\phi_v(x)$ with $v(x,t)$ in \eqref{eq:velocity_weak_pen} and $\phi_{\sigma}(x)$ with ${\sigma}(x,t)$ in \eqref{eq:stress_weak_pen}, and integrate by parts the spatial derivative term in \eqref{eq:velocity_weak_pen} we have
\begin{equation}\label{eq:velocity_weak_pen_1}
\begin{split}
\small
&\sum_{k = 1}^{K}{\left(\int_{x_k}^{x_{k+1}}\left({\rho(x)} v(x,t)\frac{\partial v(x, t)}{\partial t}  + \sigma(x,t)\frac{\partial v(x, t)}{\partial x} \right) dx -v(x_{k+1}, t)\sigma(x_{k+1}, t)  \right)}\\
+& \sum_{k = 1}^{K}{\left(v(x_k, t)\sigma(x_k, t)+ v(x_k, t) F(x_k, t) +  v(x_{k+1}, t)G(x_{k+1}, t)  \right)} = 0, 
\end{split}
\end{equation}
\begin{equation}\label{eq:stress_weak_pen_1}
\begin{split}
\small
&\sum_{k = 1}^{K}\left(\int_{x_k}^{x_{k+1}}\left(\frac{1}{\mu(x)}\sigma(x,t)\frac{\partial \sigma(x, t)}{\partial t}  - \sigma(x,t)\frac{\partial v(x, t)}{\partial x} \right) dx \right) \\
&-\sum_{k = 1}^{K}\left( \frac{{\sigma}(x_{k}, t)}{Z_s(x_{k})}F(x_{k}, t) -   \frac{{\sigma}(x_{k+1}, t)}{Z_s(x_{k+1})}G(x_{k+1}, t)\right) = 0.
\end{split}
\end{equation}
Thus, summing  \eqref{eq:velocity_weak_pen_1} and \eqref{eq:stress_weak_pen_1} together,  the interior terms involving spatial derivatives cancel leaving only the boundary terms, having 
\begin{equation}\label{eq:identity_pen_3}
\begin{split}
&\frac{d}{dt}\left[\sum_{k=1}^{K} \frac{1}{2}\int_{x_k}^{x_{k+1}}{\left({\rho(x)} v^2(x, t) + \frac{1}{\mu(x)}\sigma^2(x, t)\right) dx}\right] = \sum_{k = 1}^{K}{\left(v(x_{k+1}, t)\sigma(x_{k+1}, t)\right) }  \\
&-\sum_{k = 1}^{K}{\left(v(x_k, t)\sigma(x_k, t)+ v(x_k, t) F(x_k, t) +  v(x_{k+1}, t)G(x_{k+1}, t) \right) } \\
&+\sum_{k = 1}^{K}\left( \frac{{\sigma}(x_{k}, t)}{Z_s(x_{k})}F(x_{k}, t) -   \frac{{\sigma}(x_{k+1}, t)}{Z_s(x_{k+1})}G(x_{k+1}, t)\right).
\end{split}
\end{equation}
Note that
\begin{equation}\label{eq:identity_pen_1}
\begin{split}
\small
&v(x_k,t) F(x_k,t) + v(x_k,t) \sigma(x_k,t)  - \frac{1}{Z_s(x_k) }\sigma(x_k,t)  F(x_k,t) \\
&= \frac{1}{Z_s(x_k) }\left(|F(x_k,t)|^2 + p^2(x_k,t) - \widehat{q}^2(x_k,t)\right),
\end{split}
\end{equation}
\begin{equation}\label{eq:identity_pen_2}
\begin{split}
\small
&v(x_{k+1},t) G(x_{k+1}, t) - v(x_{k+1},t) \sigma(x_{k+1},t)  + \frac{1}{Z_s(x_{k+1}) }\sigma(x_{k+1},t)  G(x_{k+1},t) \\
&= \frac{1}{Z_s(x_{k+1}) }\left(|G(x_{k+1}, t)|^2 + q^2(x_{k+1}, t) - \widehat{p}^2(x_{k+1}, t)\right).
\end{split}
\end{equation}
 If we define,
\begin{equation}\label{eq:continuous_energy1}
E^k(t) = \frac{1}{2}\int_{x_k}^{x_{k+1}}{\left({\rho(x)} v^2(x, t) + \frac{1}{\mu(x)}\sigma^2(x, t)\right) dx},
\end{equation}
then we have $E(t) = \sum_{k=1}^{K}E^k(t)$. Thus,  using \eqref{eq:identity_pen_1}-\eqref{eq:identity_pen_2} in the right hand side of \eqref{eq:identity_pen_3} gives
\begin{equation}\label{eq:identity_pen_4}
\begin{split}
&\frac{d}{dt} E(t) = -\sum_{k = 1}^{K}{\left( \frac{1}{Z_s(x_k) }\left(|F(x_k,t)|^2 + p^2(x_k,t) - \widehat{q}^2(x_k,t)\right)\right) }  \\
&-\sum_{k = 1}^{K}{\left(\frac{1}{Z_s(x_{k+1}) }\left(|G(x_{k+1}, t)|^2 + q^2(x_{k+1}, t) - \widehat{p}^2(x_{k+1}, t)\right)\right) }. 
\end{split}
\end{equation}
Using the  identities \eqref{eq:identity_1_bc}--\eqref{eq:identity_3_bc} and \eqref{eq:identity_1}--\eqref{eq:identity_3},  with
\[ 
\widehat{\sigma}(x_k)= \frac{\alpha(x_k) }{\eta(x_k)  + \alpha(x_k) }\Phi(x_k)  , \quad   \llbracket \widehat{v}(x_k) \rrbracket= \frac{1}{\eta(x_k) + \alpha(x_k)}\Phi(x_k),
\]
in the right hand side of \eqref{eq:identity_pen_4} gives the energy identity \eqref{eq:weak_energy} \end{proof}

Since $|r_0|\le 1$, $|r_L|\le 1$ and $ \widehat{\sigma} \llbracket \widehat{v} \rrbracket = \frac{\alpha}{(\eta + \alpha)^2}\Phi^2 \ge 0$, then the boundary terms in the right hand side of \eqref{eq:weak_energy} are negative semi-definite. The term $ \widehat{\sigma} \llbracket \widehat{v} \rrbracket = \frac{\alpha}{(\eta + \alpha)^2}\Phi^2 \ge 0$ represents the rate of work done by friction at the interface, which is dissipated as heat.
Note again that  the flux fluctuations vanish identically  $G(x_{k+1}, t) \equiv 0$,  $F(x_{k}, t) \equiv  0$ for exact solutions, that satisfy the PDE and the boundary and interface conditions, \eqref{eq:BC}  and \eqref{eq:physical_interface}.  Thus, the energy equation \eqref{eq:weak_energy} is completely identical to \eqref{eq:continuous_energy_rate_1}. At the limit $\alpha \to \infty \iff \widehat{\sigma}(x_k)\llbracket \widehat{v}(x_k) \rrbracket \to 0$, we obtain the energy identity \eqref{eq:continuous_energy_rate}.
However, when numerical approximations are introduced the numerical solutions will be accurate up to the truncation error, and $G(x_{k+1}, t) \ne 0$,  $F(x_{k}, t) \ne  0$.  The flux fluctuations, $G(x_{k+1}, t)$,  $F(x_{k}, t)$ will be proportional to the truncation error and will introduce some  numerical dissipation.  However, the numerical dissipation will vanish in the limit of mesh refinement,  $\Delta{x}_k \to 0$ with $\Delta{x}_k = x_{k+1} - x_k$. The remaining terms in the right hand side of \eqref{eq:weak_energy} match exactly the physical energy rate given by the boundary condition  \eqref{eq:BC} and  interface condition \eqref{eq:physical_interface}.

\subsection{The  Galerkin approximation}
Since $(\phi_v(x), \phi_\sigma(x)) \in \mathbb{L}^2(0, L)$ we can selectively choose $(\phi_v(x), \phi_\sigma(x))$ to be nonzero in one element, $ [x_k, x_{k+1}]$, having
\begin{equation}\label{eq:elemental_weak_form_velocity}
\begin{split}
\int_{x_k}^{x_{k+1}} & \left({\rho(x)} \phi_v(x)\frac{\partial v(x, t)}{\partial t}  - \phi_v(x)\frac{\partial \sigma(x, t)}{\partial x} \right) dx + \phi_v(x_k) F(x_k, t) +  \phi_v(x_{k+1})G(x_{k+1}, t) = 0, 
\end{split}
\end{equation}
\begin{equation}\label{eq:elemental_weak_form_stress}
\begin{split}
\int_{x_k}^{x_{k+1}}  &  \left(\frac{1}{\mu(x)}\phi_{\sigma}(x)\frac{\partial \sigma(x, t)}{\partial t}  - \phi_{\sigma}(x)\frac{\partial v(x, t)}{\partial x} \right) dx - \frac{\phi_{\sigma}(x_{k})}{Z_s(x_{k})}F(x_{k}, t) +   \frac{\phi_{\sigma}(x_{k+1})}{Z_s(x_{k+1})}G(x_{k+1}, t) = 0.
\end{split}
\end{equation}

Next, we map the element $[x_k, x_{k+1}]$ to a reference element $\xi \in [-1, 1]$ by the linear transformation 
\begin{align}\label{eq:transf}
x = x_k + \frac{\Delta{x}_k}{2}\left(1 + \xi \right), \quad \Delta{x}_k = x_{k+1} - x_k.
\end{align}
Introducing the linear tranformation \eqref{eq:transf} in the elemental weak form \eqref{eq:elemental_weak_form_velocity}--\eqref{eq:elemental_weak_form_stress}, we have
\begin{align}\label{eq:transf_elemental_weak_form_velocity}
\small
\frac{\Delta{x}_k}{2}\int_{-1}^{1}{\rho(\xi)} \phi_v(x)\frac{\partial v(\xi, t)}{\partial t}d\xi  &= \int_{-1}^{1}\phi_v(\xi)\frac{\partial \sigma(\xi, t)}{\partial \xi} d\xi 
-  \phi_v(-1) F(-1, t) -  \phi_v(1)G(1, t), 
\end{align}
\begin{align}\label{eq:transf_elemental_weak_form_stress}
\small
\frac{\Delta{x}_k}{2}\int_{-1}^{1}\frac{1}{\mu(\xi)}\phi_{\sigma}(\xi)\frac{\partial \sigma(\xi, t)}{\partial t} d\xi  &= \int_{-1}^{1}\phi_{\sigma}(\xi)\frac{\partial v(\xi, t)}{\partial \xi} d\xi  + \frac{\phi_{\sigma}(-1)}{Z_s(-1)}F(-1, t) -   \frac{\phi_{\sigma}(1)}{Z_s(1)}G(1, t) .
\end{align}
Inside the transformed  element  $\xi \in [-1, 1]$, approximate the solution  and material parameters by a polynomial interpolant,  and write 
\begin{equation}\label{eq:variables_elemental}
v^k(\xi, t) = \sum_{j = 1}^{N+1}v_j^k(t) \mathcal{L}_j(\xi), \quad \sigma^k(\xi, t)  = \sum_{j = 1}^{N+1}\sigma_j^k(t) \mathcal{L}_j(\xi),
\end{equation}
\begin{equation}\label{eq:material_elemental}
\rho^k(\xi) = \sum_{j = 1}^{N+1}\rho_j^k \mathcal{L}_j(\xi), \quad \mu^k(\xi) = \sum_{j = 1}^{N+1}\mu_j^k \mathcal{L}_j(\xi),
\end{equation}
where $ \mathcal{L}_j$ is the $j$th interpolating polynomial of degree $N$. If we consider  nodal basis then the interpolating polynomials satisfy $ \mathcal{L}_j(\xi_i) = \delta_{ij}$. 
The interpolating nodes $\xi_i$, $i = 1, 2, \dots, N+1$ are the nodes of a Gauss quadrature with
\begin{equation}
 \sum_{i = 1}^{N+1} f(\xi_i)w_i \approx \int_{-1}^{1}f(\xi) d\xi,
\end{equation}
where $w_i$ are quadrature weights.
We will only use quadrature rules that are exact for all polynomial integrand $f(\xi)$ of degree $\le 2N-1$. Admissible  candidates are Gauss-Lobatto quadrature rule with GLL nodes and Gauss-Legendre quadrature rule with GL nodes. Note that, boundary points $\xi = -1, 1$ are part of   GLL quadrature nodes  while    boundary points $\xi = -1, 1$ are not  part of  GL quadrature nodes. 
The material parameters are interpolated exactly at the quadrature nodes.

 We now make a classical  Galerkin approximation by choosing test functions $(\phi_v(\xi), \phi_\sigma(\xi))$ in the same space as the basis functions, so that the residual is orthogonal to the space of test functions. 
 
Introduce the weighted elemental mass matrix $W^N(a)$ and the stiffness matrix $Q^N $ defined by
{\small
\begin{equation}
W^N_{ij}(a) = \sum_{m = 1}^{N+1} w_m \mathcal{L}_i(\xi_m)  {\mathcal{L}_j(\xi_m)} a(\xi_m), \quad Q^N_{ij} = \sum_{m = 1}^{N+1} w_m \mathcal{L}_i(\xi_m)  {\mathcal{L}_j^{\prime}(\xi_m)}.
\end{equation}
}
For all positive coefficients $a(\xi) > 0$ and quadrature weights $w_m > 0$, the mass matrix is symmetric positive definite, $W^N(a)  = \left(W^N(a)\right)^T > 0$. 
 If we consider  nodal basis $\mathcal{L}_j(\xi)$ with $ \mathcal{L}_j(\xi_i) = \delta_{ij}$, then the mass matrix is diagonal with
 \begin{equation}
W^N_{ij}(a) = w_j a(\xi_j) \delta_{ij}.
\end{equation}
Note that integration-by-parts yields
\begin{equation}\label{eq:ibp_property}
\small
\int_{-1}^{1}\mathcal{L}_i(\xi)  {\mathcal{L}_j^{\prime}(\xi)} d\xi = -\int_{-1}^{1}\mathcal{L}_i^{\prime}(\xi)  {\mathcal{L}_j(\xi)} d\xi + {\mathcal{L}_j(1)}{\mathcal{L}_i(1)} - {\mathcal{L}_j(-1)}{\mathcal{L}_i(-1)}.
\end{equation}
Thus using the fact that the quadrature rule is exact  for all polynomial  intergrand of degree $\le 2N-1$ and defining the transpose of the stiffness matrix
\[
\left(Q^N\right)^T_{ij} = \sum_{m = 1}^{N+1} w_m \mathcal{L}_i^{\prime}(\xi_m)  {\mathcal{L}_j(\xi_m)} = \int_{-1}^{1}\mathcal{L}_i^{\prime}(\xi)  {\mathcal{L}_j(\xi)} d\xi ,
\]
  implies that
\begin{equation}\label{eq:sbp_property_a}
Q^N_{ij} + \left(Q^N\right)^T_{ij} =  B^N_{ij},
\end{equation}
where 
\begin{equation}\label{eq:sbp_property_b}
B^N_{ij} = {\mathcal{L}_j(1)}{\mathcal{L}_i(1)} - {\mathcal{L}_j(-1)}{\mathcal{L}_i(-1)}.
\end{equation}
Equation \eqref{eq:sbp_property_a}-\eqref{eq:sbp_property_b} is the discrete equivalence of the integration-by-parts property \eqref{eq:ibp_property}.
If boundary points $\xi = -1,1$ are quadrature nodes and we consider nodal bases  with $ \mathcal{L}_j(\xi_i) = \delta_{ij}$ then we have $B^N = \text{diag}\left([-1, 0,0, \dots, 0, 1]\right).$ 
In the finite difference literature \cite{DelReyFernandezBoomZingg2014, DuruandDunham2016} equation \eqref{eq:sbp_property_a}-\eqref{eq:sbp_property_b} is analogous to the so-called  summation-by-parts (SBP) property.

The elemental degrees of freedom to be evolved are arranged as vectors of length $N+1$
\[
\boldsymbol{v}^k(t) = [v^k_1(t) , v^k_2(t) , \dots, v^k_{N+1}(t) ]^T, \quad \boldsymbol{\sigma}^k(t)  = [{\sigma}^k_1(t) , {\sigma}^k_2(t) , \dots, {\sigma}^k_{N+1}(t) ]^T.
\]
The evolution equations for the elemental degrees of freedom are a semi-discrete approximation of the IBVP, \eqref{eq:elastic_1D} with \eqref{eq:boundary_conditions} or \eqref{eq:BC} and \eqref{eq:physical_interface}, which can be written as a linear system of ODEs
\begin{equation}\label{eq:elemental_pde1}
\begin{split}
\small
\frac{\Delta{x}_k}{2} W^N ({\boldsymbol{\rho}}^{k}) \frac{d \boldsymbol{v}^k( t)}{ d t} &= Q^N \boldsymbol{\sigma}^k( t) - \boldsymbol{e}_{1}F^k(-1, t)- \boldsymbol{e}_{N+1}G^k(1, t),
\end{split}
\end{equation}
\begin{equation}\label{eq:elemental_pde2}
\begin{split}
\small
\frac{\Delta{x}_k}{2} W^N \left({1}/{\boldsymbol{\mu}^{k}}\right) \frac{d \boldsymbol{\sigma}^k( t)}{ d t} &= Q^N \boldsymbol{v}^k( t)  + \boldsymbol{e}_{1}\frac{1}{Z_{s}^{k}(-1)}F^k(-1, t)- \boldsymbol{e}_{N+1}\frac{1}{Z_{s}^{k}(1)}G^k(1, t),
\end{split}
\end{equation}
where 
\begin{align*}
\boldsymbol{e}_{1} = [ \mathcal{L}_1(-1), \mathcal{L}_2(-1), \dots,  \mathcal{L}_{N+1}(-1) ]^T, \quad  \boldsymbol{e}_{N+1} = [ \mathcal{L}_1(1), \mathcal{L}_2(1), \dots,  \mathcal{L}_{N+1}(1) ]^T,
\end{align*}
and
\begin{align*}
G^k(1, t):= \frac{Z_{s}^{k}(1)}{2} \left(v^{k}(1, t)-\widehat{v}^{k}(1, t) \right) + \frac{1}{2}\left(\sigma^{k}(1, t)- \widehat{\sigma}^{k}(1, t)\right), 
\end{align*}
\begin{align*}
F^{k}(-1, t):= \frac{Z_{s}^{k}(-1)}{2} \left(v^{k}(-1, t)-\widehat{v}^{k}(-1, t) \right) - \frac{1}{2}\left(\sigma^{k}(-1, t)- \widehat{\sigma}^{k}(-1, t)\right).
\end{align*}
Equations \eqref{eq:elemental_pde1}-\eqref{eq:elemental_pde2} are a discontinuous Galerkin approximation of the IBVP, \eqref{eq:elastic_1D} with  \eqref{eq:BC} and \eqref{eq:physical_interface}. 

The hat-variables,  at the element  boundaries, $\xi = -1, 1$, are computed as outlined in sections \ref{subsec:boundary_hat_variables} and \ref{subsec:interface_hat_variables}. 
The only difference is that  instead of the  continuous solutions used in sections \ref{subsec:boundary_hat_variables} and \ref{subsec:interface_hat_variables}, the numerical boundary/interface data for the characteristics are generated using the elemental polynomial approximations, $v^k (\xi, t), \sigma^k(\xi, t)$,   and the approximated material parameters $\rho^k (\xi), \mu^k(\xi)$,    defined in \eqref{eq:variables_elemental}--\eqref{eq:material_elemental}, and evaluated at the boundaries, at $\xi = -1, 1$. However, as before, the discrete hat-variables satisfy the same algebraic identities, \eqref{eq:identity_1_bc}--\eqref{eq:identity_3_bc} and \eqref{eq:identity_1}--\eqref{eq:identity_3},  as the continuous counterparts. 

The system of ODEs \eqref{eq:elemental_pde1}-\eqref{eq:elemental_pde2} is a semi-discrete approximation of the IBVP, \eqref{eq:elastic_1D} with \eqref{eq:boundary_conditions} or \eqref{eq:BC}. For the semi-discrete approximation \eqref{eq:elemental_pde1}--\eqref{eq:elemental_pde2}, the flux fluctuations  will vanish identically, $F^{k}(-1, t) \to 0$, $G^k(1, t) \to 0$, only in the limit of mesh refinement, $\Delta{x}_k \to 0$.

\section{Stability}
In this section, we will prove that the semi-discrete approximation \eqref{eq:elemental_pde1}--\eqref{eq:elemental_pde2} is asymptotically stable. We will derive discrete energy equation analogous to the continuous energy equation \eqref{eq:weak_energy}. To begin, define the elemental discrete energy
\begin{equation}\label{eq:local_energy}
\small
 \mathcal{E}^k(t) = \frac{\Delta{x}_k}{2}\left(\frac{1}{2}\left(\boldsymbol{v}^k( t)\right)^TW^N ({\rho}^{k})  \boldsymbol{v}^k( t) + \frac{1}{2}\left(\boldsymbol{\sigma}^k( t)\right)^TW^N ({1/\mu}^{k})  \boldsymbol{\sigma}^k( t)\right).
\end{equation}
If we consider a nodal polynomial basis $\mathcal{L}_j(\xi)$ with $ \mathcal{L}_j(\xi_i) = \delta_{ij}$, then the mass matrix is diagonal and we have
{\small
\begin{equation}\label{eq:local_energy2}
\small
 \mathcal{E}^k(t) = \frac{\Delta{x}_k}{2}\sum_{j=1}^{N+1}\left(\frac{w_j}{2} \left( {\rho}_j^{k}|{v}_j^k( t)|^2 +  \frac{1 }{\mu^k_j}|\sigma^k_j( t)|^2\right) \right).
\end{equation}
}

If  the semi-discrete energy  $\mathcal{E}(t) = \sum_{k=1}^{K} \mathcal{E}^k(t) $ is never permitted to grow in time for any $\Delta{x} > 0$,  we say that the semi-discrete approximation \eqref{eq:elemental_pde1}-\eqref{eq:elemental_pde2} is asymptotically  stable. We will make this statement more precise with the definition
\begin{definition}\label{def:asymtotic_stability}
Let  $\mathcal{E}(t) = \sum_{k=1}^{K} \mathcal{E}^k(t) $  denote the global semi-discrete energy. The semi-discrete approximation \eqref{eq:elemental_pde1}-\eqref{eq:elemental_pde2} is asymptotically stable if 
\begin{align}
\frac{d}{dt} \mathcal{E}(t) \le 0, \quad \forall \Delta{x} > 0.
\end{align}
\end{definition}

Our second main result is the following theorem:
\begin{theorem}\label{theorem:main_2}
The semi-discrete approximation \eqref{eq:elemental_pde1}-\eqref{eq:elemental_pde2} satisfies the energy equation
{\small
\begin{equation}\label{eq:energy_equation_discrete}
\begin{split}
\frac{d}{dt}\mathcal{E}(t) &= -\sum_{k=1}^{K}\left(\frac{1}{Z_{s}^k(-1)}|F^k(-1, t)|^2 + \frac{1}{Z_{s}^k(1)}|G^k(1, t)|^2\right)   -\sum_{k=2}^{K} \frac{\alpha^k}{\left(\eta^k + \alpha^k\right)^2}|\Phi^k|^2 \\
& - \frac{1-r_0^2}{Z_{s}^1(-1)}|p_0|^2 -\frac{1-r_L^2}{Z_{s}^K(1)}|q_L|^2,
\end{split}
\end{equation}
}
with $ \mathcal{E}(t) = \sum_{k=1}^{K} \mathcal{E}^k(t) $,  and 
\begin{align*}
p_0 = \frac{1}{2}\left(Z_s^1(-1)v^1(-1, t) + \sigma^1(-1, t)\right), \quad q_L = \frac{1}{2}\left(Z_s^K(1) v^K(1, t) - \sigma^K(1, t)\right).
\end{align*}
\end{theorem}
\begin{proof}
The derivation of the energy equation \eqref{eq:energy_equation_discrete} follows from standard energy method and calculations. That is, from the left we multiply equation \eqref{eq:elemental_pde1} by $[{v}^k_1( t), {v}^k_2( t), \dots, {v}^k_{N+1}( t)]^T$ and  \eqref{eq:elemental_pde2} by $[{\sigma}^k_1( t), {\sigma}^k_2( t), \dots, {\sigma}^k_{N+1}( t)]^T$.  We use   the discrete integration-by-parts property \eqref{eq:sbp_property_a}--\eqref{eq:sbp_property_b} in the velocity equation \eqref{eq:elemental_pde1} only, having
\begin{equation}\label{eq:elemental_pde1_energy}
\begin{split}
\small
&\frac{\Delta{x}_k}{2} \boldsymbol{v}^k( t)^TW^N ({\rho}^{k}) \frac{d \boldsymbol{v}^k( t)}{ d t} = -\boldsymbol{v}^k( t)^T\left(Q^N\right)^T \boldsymbol{\sigma}^k( t) -{v}^k(-1, t){\sigma}^k(-1, t) \\
& + {v}^k(1, t){\sigma}^k(1, t) - {v}^k(-1, t)F^k(-1, t)- {v}^k(1, t)G^k(1, t),
\end{split}
\end{equation}
\begin{equation}\label{eq:elemental_pde2_energy}
\begin{split}
\small
&\frac{\Delta{x}_k}{2} \boldsymbol{\sigma}^k( t)^TW^N \left({1}/{{\mu}^{k}}\right) \frac{d \boldsymbol{\sigma}^k( t)}{ d t} = \boldsymbol{\sigma}^k( t)^TQ^N \boldsymbol{v}^k( t) \\
& + \frac{1}{Z_{s}^{k}(-1)}{\sigma}^k( -1, t)F^k(-1, t)- \frac{1}{Z_{s}^{k}(1)}{\sigma}^k( 1, t)G^k(1, t).
\end{split}
\end{equation}
Then summing the products, \eqref{eq:elemental_pde1_energy} and \eqref{eq:elemental_pde2_energy} together, in the right hand side, the interior terms cancel, leaving   the element boundary terms only, having
\begin{equation}\label{eq:local_energy_pde}
\begin{split}
\small
 &\frac{d}{dt}\left[\frac{\Delta{x}_k}{2}\left(\frac{1}{2}\left(\boldsymbol{v}^k( t)\right)^TW^N ({\rho}^{k})  \boldsymbol{v}^k( t) + \frac{1}{2}\left(\boldsymbol{\sigma}^k( t)\right)^TW^N ({1/\mu}^{k})  \boldsymbol{\sigma}^k( t)\right)\right] = \\
 &-{v}^k(-1, t){\sigma}^k(-1, t) + {v}^k(1, t){\sigma}^k(1, t) - {v}^k(-1, t)F^k(-1, t)- {v}^k(1, t)G^k(1, t)\\
 & + \frac{1}{Z_{s}^{k}(-1)}{\sigma}^k( -1, t)F^k(-1, t)- \frac{1}{Z_{s}^{k}(1)}{\sigma}^k( 1, t)G^k(1, t).
 \end{split}
\end{equation}
In the left hand side of \eqref{eq:local_energy_pde}, we recognize the  elemental semi-discrete energy   $\mathcal{E}^k(t)$ defined in \eqref{eq:local_energy}.
As in \eqref{eq:identity_pen_1}-\eqref{eq:identity_pen_2}, note again that
\begin{equation}\label{eq:identity_pen_11}
\begin{split}
\small
&v^k(-1,t) F^k(-1,t) + v^k(-1,t) \sigma^k(-1,t)  - \frac{1}{Z_s^k(-1) }\sigma^k(-1,t)  F^k(-1,t) \\
&=\frac{1}{Z_s^k(-1) }\left(|F^k(-1,t)|^2 + \left(p^k(-1, t)\right)^2 - \left(\widehat{q}^k(-1, t)\right)^2\right),
 \end{split}
\end{equation}
\begin{equation}\label{eq:identity_pen_22}
\begin{split}
\small
&v^k(1,t) G^k(1, t) - v^k(1,t) \sigma^k(1,t)  + \frac{1}{Z_s^k(1) }\sigma^k(1,t)  G^k(1,t)  \\
&=\frac{1}{Z_s^k(1) }\left(|G^k(1, t)|^2 + \left(q^k(1, t)\right)^2 - \left(\widehat{p}^k(1, t)\right)^2\right).
 \end{split}
\end{equation}
Thus, using \eqref{eq:identity_pen_11}-\eqref{eq:identity_pen_22} in the right hand side of \eqref{eq:local_energy_pde}, yields
\begin{equation}\label{eq:local_energy_pde2}
\begin{split}
\small
 \frac{d}{dt}\mathcal{E}^k(t)  = &-\frac{1}{Z_s^k(-1) }\left(|F^k(-1,t)|^2 + \left(p^k(-1, t)\right)^2 - \left(\widehat{q}^k(-1, t)\right)^2\right) \\
 &- \frac{1}{Z_s^k(1) }\left(|G^k(1, t)|^2 + \left(q^k(1, t)\right)^2 - \left(\widehat{p}^k(1, t)\right)^2\right).
 \end{split}
\end{equation}
 Adding contributions from all elements and using \eqref{eq:sol_linear_equation}, that is
 \[
 \widehat{\sigma}^{k+1}(-1, t) = \widehat{\sigma}^{k}(1, t) = \widehat{\sigma}^k = \frac{\alpha^k}{\eta^k + \alpha^k} \Phi^k, \quad   \widehat{v}^{k+1}(-1, t) - \widehat{v}^{k}(1, t):= \llbracket \widehat{v}^k \rrbracket = \frac{1}{\eta^k + \alpha^k} \Phi^k,
 \]
  and  the identities \eqref{eq:identity_2_bc}--\eqref{eq:identity_3_bc}  and  \eqref{eq:identity_2}--\eqref{eq:identity_3}  
 gives the energy equation \eqref{eq:energy_equation_discrete}.
 \end{proof} 
  
  The  energy equation \eqref{eq:energy_equation_discrete} is completely analogous to the continuous equation \eqref{eq:weak_energy} and \eqref{eq:continuous_energy_rate_1}. 
  Note that  the quantity  in the right hand side of \eqref{eq:energy_equation_discrete} are surface terms, and their units match energy-rate per surface area. Thus, we have generated numerical  data in a manner that is consistent with physical laws and  enforced element boundary data using characteristics, the natural carrier of information in the system.  Note that {\small $\widehat{\sigma}^k {\lJump  \widehat{v}^k  \rJump} =  \frac{\alpha^k}{(\eta^k + \alpha^k)^2}|\Phi^k|^2 \to 0$}, for {\small $\alpha^k \to \infty$ or $\alpha^k \to 0$}. This implies that the spectral radius of the discrete operator has an upper bound which is independent of $\alpha^k \ge 0$. If we had used characteristics to directly enforce the physical condition \eqref{eq:physical_interface},  we will have  {\small$ \sigma^k \lJump  {v}^k  \rJump  = \alpha^k \lJump  {v}^k  \rJump ^2\ge 0$}. The semi-discrete approximation  will yield an energy estimate, however, it will potentially introduce artificial numerical stiffness, for $\alpha^k \gg 1$, which will require implicit time integration, for practical problems.   
 Note that the energy equation \eqref{eq:energy_equation_discrete}  is valid for both nodal and modal polynomial basis,  and for  Gauss-Lobatto-Legendre nodes, with boundary points as quadrature nodes, and Gauss-Legendre nodes, where boundary points are not quadrature nodes.
 \begin{remark}
 Theorems \ref{theorem:main_1} and \ref{theorem:main_2} prove both asymptotic stability and robustness of the method.
 Note that all  internal DG element faces are frictional interfaces. In principle, we can allow all adjacent elements to slide against each other, and the frictional slip motion  governed by a nonlinear friction law.
 This will not affect stability, but will increase energy decay rate, due to work done by friction which is dissipated as heat.
 One major outcome of our approach is that it yields a unified provable stable and robust adaptive DG framework for the numerical treatment of  nonlinear frictional source terms accommodating slip, classical DG inter-element interfaces where slip is not permitted, and external well-posed boundary conditions modeling various geophysical phenomena.
 \end{remark}

The analysis here focuses on a 1D model problem, however with limited modifications the results can be extended to multidimensional (2D and 3D) tensor product DG method approximations of the elastic wave equation on quadrilateral and hexahedral meshes, and also on triangular and tetrahedral  meshes.

\section{Numerical experiments}
Here, we perform numerical experiments to verify numerical stability and accuracy.  Lagrange polynomial bases are used with GLL and GL quadrature nodes, separately. Numerical solutions are evolved in time using the high order  ADER  scheme \cite{DumbserPeshkovRomenski, delaPuenteAmpueroKaser2009, PeltiesdelaPuenteAmpueroBrietzkeKaser2012}  of the same order of accuracy with the spatial discretization. 
Thus, for polynomial approximations of degree $N$, we will expect optimal asymptotic convergence rate of $N+1$. We will proceed later to a 2D model problem, make comparisons with the Rusanov flux and verify accuracy for Rayleigh surface waves. Finally, we  will present numerical experiments, in 2D, demonstrating the extension of our method to curvilinear elements and potentials for propagating ruptures on dynamically adaptive meshes.

\subsection{One space dimension}
We will now present numerical examples in 1D. We will begin with wave propagation in a heterogeneous medium, and lock all interior element boundaries, {\small $\alpha \to \infty$}.
Next we consider a dynamic rupture model, where a nonlinear frictional fault  is present. The fault will be governed by a slip-weakening friction law \cite{Andrews1985}.
\subsubsection{Wave propagation in a heterogeneous medium}
We consider a 1D domain, $0 \le x \le L = 10$ km, with the heterogeneous shear wave velocity profile   $c_s = c_0 + c_{\epsilon}(x)$.  The component $c_0$ is a mean velocity and the perturbation component $c_{\epsilon}(x)$ models small scale heterogeneity.  We use the mean shear wave velocity $c_0 = 3343$ m/s, density $\rho = 2700$ $\mathrm{kg/m^3}$, typical for crustal rocks, and set $c_{\epsilon}(x) = \epsilon\sin(n\pi x/L)$. The velocity perturbation oscillates $n=20$ times in the domain, with the amplitude $\epsilon =0.1$ km/s. Note that we can extract the shear modulus ${\mu}(x) = \rho(x)c_s^2(x)$.
  
We have chosen the initial and  boundary conditions to match the exact solution
\begin{equation}\label{eq:exact_solution}
v_e(x,t) = \cos\left({k\pi}t\right) \sin\left(\frac{n}{L}\pi x + a_0\right), \quad \sigma_e(x,t) = \frac{n}{Lk}\sin\left({k\pi}t\right) \cos\left(\frac{n}{L}\pi x + a_0\right).
\end{equation}
We chose the phase shift $a_0 = 10$, temporal wave number $k = 2$ $s^{-1}$, and the spatial wave number $n/L = 2$ ${km}^{-1}$, so that the wavelength  is in consonance with that of the small scale  heterogeneity. At the left boundary $x = 0$ we set a traction boundary condition $\sigma(0, t) = \sigma_e(0, t)$, and at the right boundary $x = L$,  we set a velocity boundary condition $v(L, t) = v_e(L, t)$. The boundary conditions are implemented weakly as discussed in previous sections.

We discretize the domain with  uniform  elements of size $ \Delta{x} = {L}/{K}$ km, where $K$ is the number of elements used. It is  important to note that the material parameters vary arbitrarily within each element. The numerical experiments shown here are performed with Lagrange  polynomial bases of degree $N = 2, 4, 6, 8, 10$. We set the time step 
\begin{align}\label{eq:explicit_time_step}
\Delta{t} = \frac{CFL}{\max_{\it{x}}(c_s( \it{x} ))(2N+1)}\Delta{x} ,  \quad \text{with}~\ {CFL} = 0.5.
\end{align}
To begin, we use  a polynomial degree $N = 4$,  and set the number of elements $K = 80$, resulting in $400$ degrees of freedom, for each unknown field, to be evolved in time. This yields 8 elements per wavelength. We evolve the solutions, for a long time, until  $t = 100$ s. 
The numerical relative error at $t = t_n$ is defined by
\begin{equation}\label{eq:L2_Error}
\mathrm{error}(t_n) =\frac{\sqrt{\sum_i\left(|v_i^n-v_e(x_i, t_n)|^2 + |\sigma_i^n-\sigma_e(x_i, t_n)|^2\right)}}{\max_{t_n}\sqrt{\left(\sum_i\left(|v_e(x_i, t_n)|^2 + |\sigma_e(x_i, t_n)|^2\right)\right)}},
\end{equation}
where $v_i^n, \sigma_i^n$ is the numerical solution and $v_e(x_i, t_n), \sigma_e(x_i, t_n)$ is the exact solution at $x = x_i$, $t = t_n$.
The numerical solution (at $t = 100$ s) superimposed with  the analytical solution, and the error are plotted in Figure \ref{fig:MMS}. Note that the  error is bounded  for the entire simulation time. Numerical errors resulting from the GLL nodes   differ from  the numerical errors  from the GL nodes by a factor 4.  The bounded  error in Figure \ref{fig:MMS} results from the discrete energy estimate \eqref{eq:energy_equation_discrete},  and  the upwind property of our numerical flux. This is consistent with the analysis in \cite{KoprivaNordstromGassner2016} for the 1D scalar advection equation.

\begin{figure}[h!]
\begin{subfigure}
    \centering
\stackunder[5pt]{\includegraphics[width=0.495\textwidth]{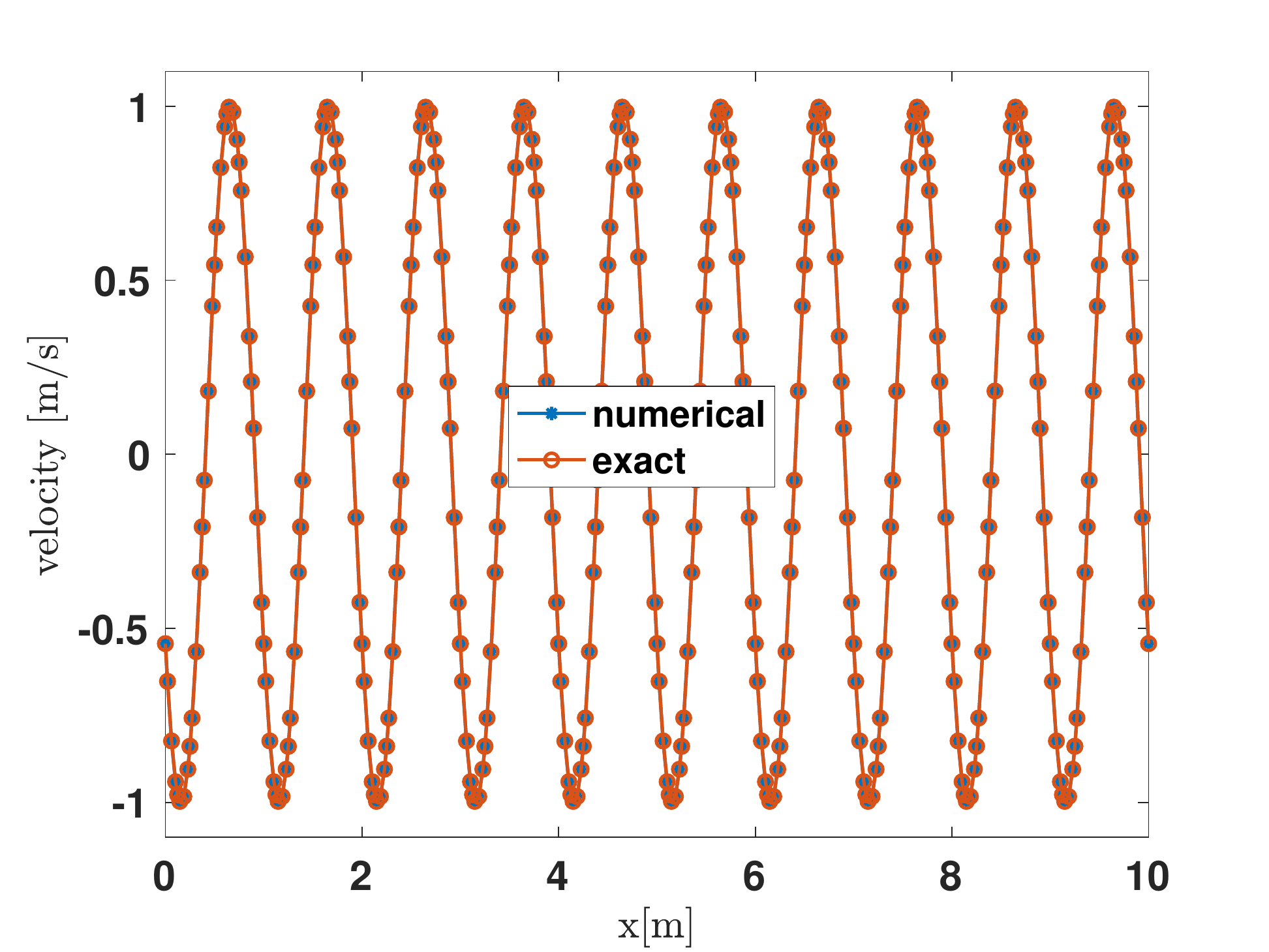}}{Numerical and exact solutions at $t = 100$ $s$.}%
\hspace{0.0cm}%
\stackunder[5pt]{\includegraphics[width=0.475\textwidth]{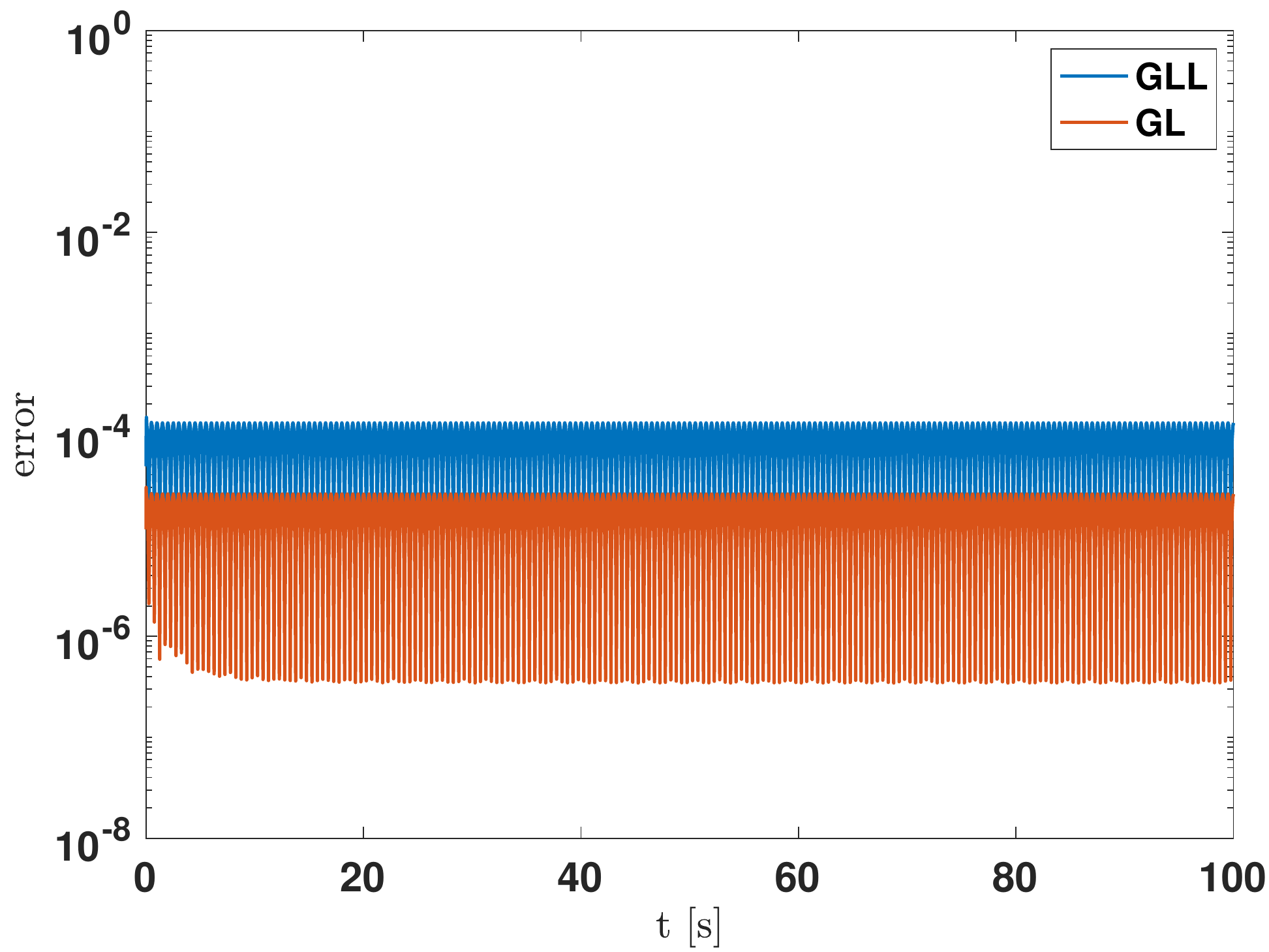}}{Time history of the  numerical error.}%
     \end{subfigure}
    \caption{ Particle velocity at $t = 100$ $s$ and time history of the  numerical error using a ${N = 4}$ (polynomial degree) and $K = 80$ number of elements.}
    \label{fig:MMS}
\end{figure}

We have run the simulation again for various resolutions. The time history of the numerical errors are shown in Figure \ref{fig:TH}. 
\begin{figure}[h!]
\begin{subfigure}
    \centering
\stackunder[5pt]{\includegraphics[width=0.475\textwidth]{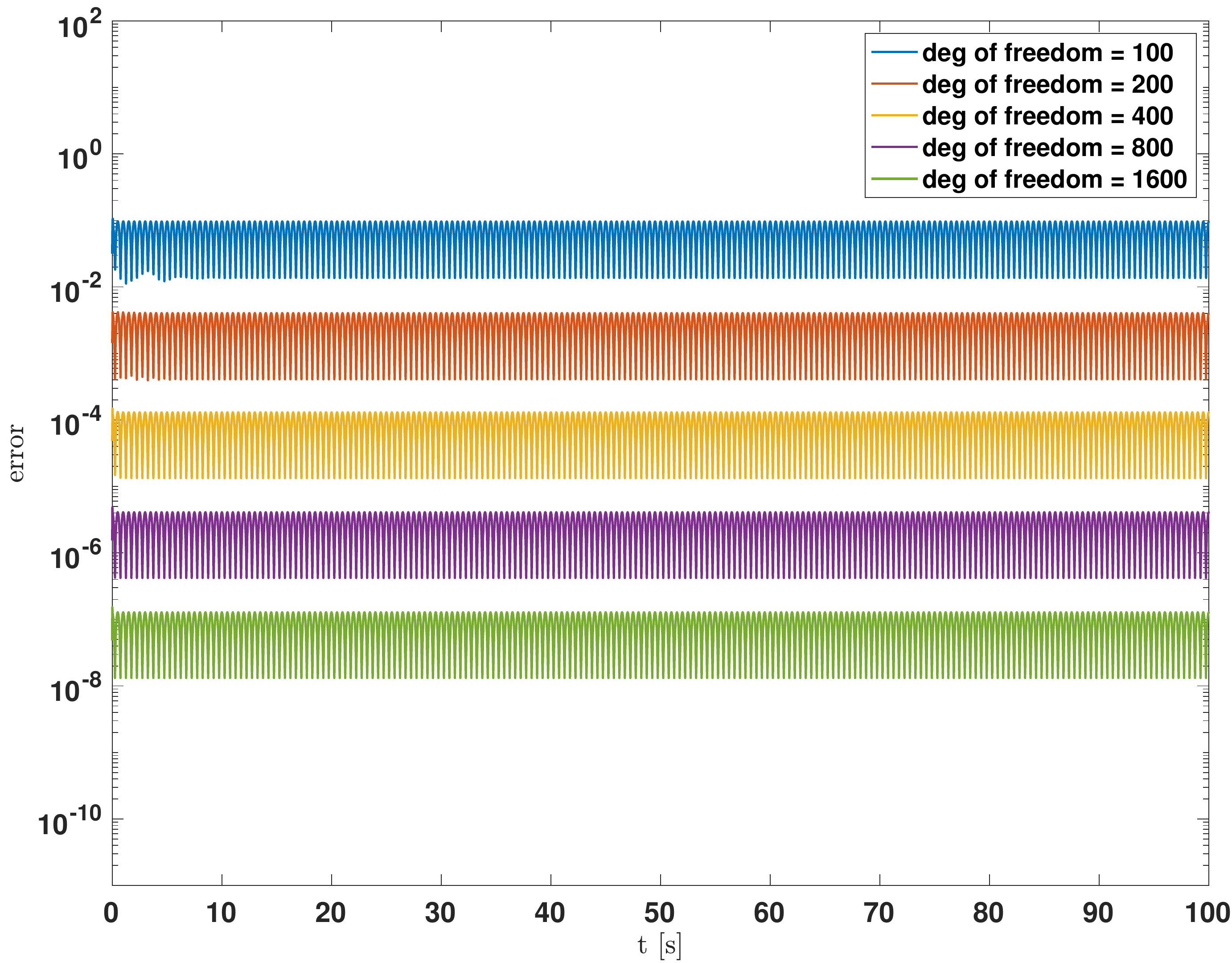}}{GLL nodes.}%
\hspace{0.0cm}%
\stackunder[5pt]{\includegraphics[width=0.475\textwidth]{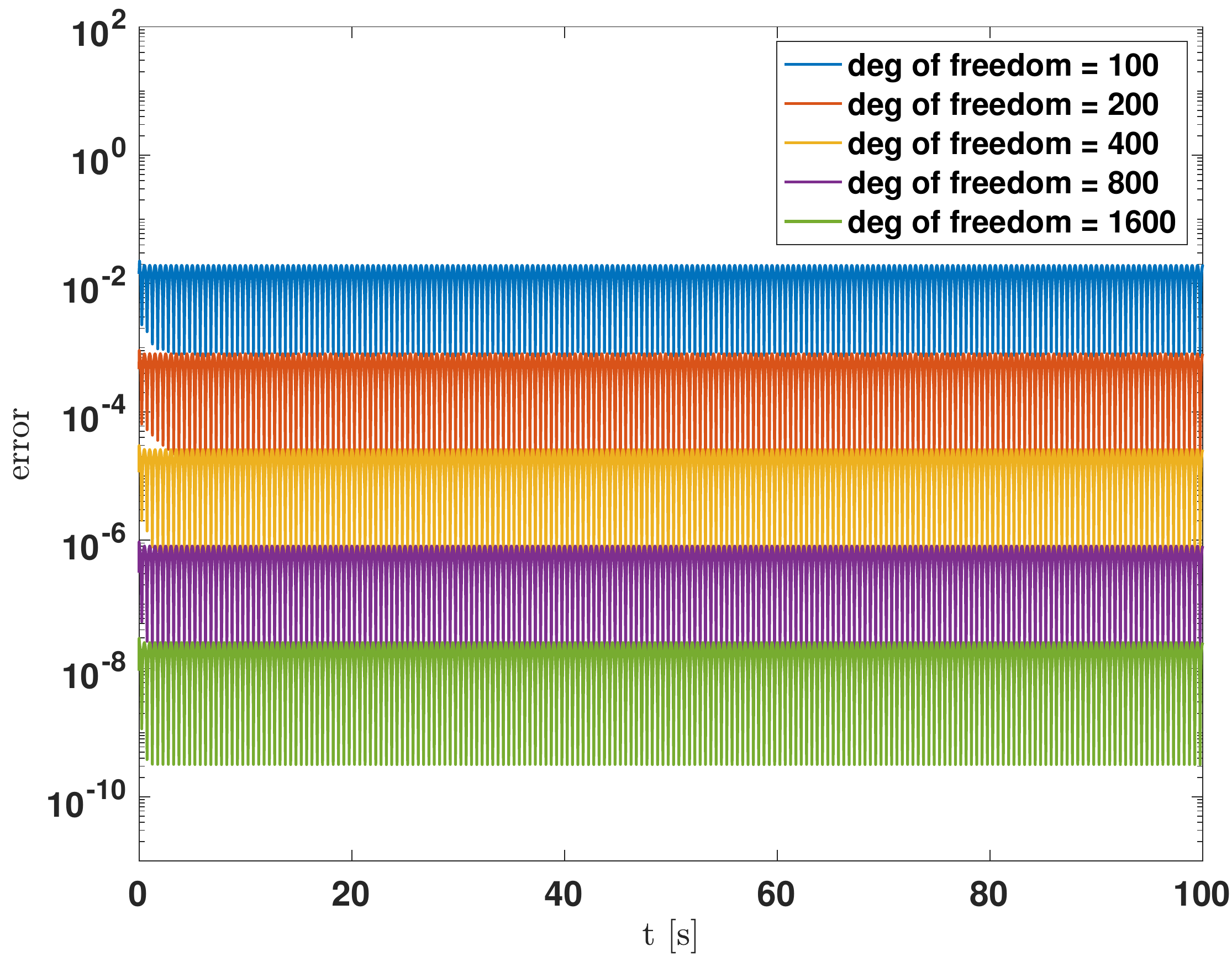}}{GL nodes.}%
     \end{subfigure}
    \caption{ Time history of numerical errors at different resolutions.}
    \label{fig:TH}
\end{figure}
\begin{table}[h!]
\centering
\begin{tabular}{c |  c | c | c | c  }
 dof & error(GLL)&  rate(GLL) & error(GL)&  rate(GL)  \\
\hline
100&9.6094e-02 & -- & {1.9066e-02} & {--} \\
200&  4.0376e-03 & 4.5729 & {8.0204e-04} & {4.5712}  \\
400&1.3010e-04 & 4.9556 & { 2.5693e-05} & {4.9642}  \\
800&4.0939e-06 & 4.9900 & {8.0751e-07} & {4.9917}  \\
1600& 1.2816e-07 & 4.9975 & {2.5271e-08} & {4.9979} \\
\hline
\end{tabular}
\caption{Numerical errors and convergence rate at $t = 100$ $s$. }
\label{tab:elastic_L0}
\end{table}
In Table \ref{tab:elastic_L0}, the numerical errors, at the final time $t = 100$ $s$, and the convergence rate are shown for different resolutions. Note that the rates of convergence is ${N+1}$, which is optimal. 
We also run the simulation with the number of elements fixed, $ K = 80$, and vary  polynomial degrees as $ N = 2, 4, 6, 8, 10.$
Spectral convergence of the discretization error is shown in Figure \ref{fig:spec_rate}.
\begin{figure}[h!]
    \centering
         {\includegraphics[width=0.5\textwidth]{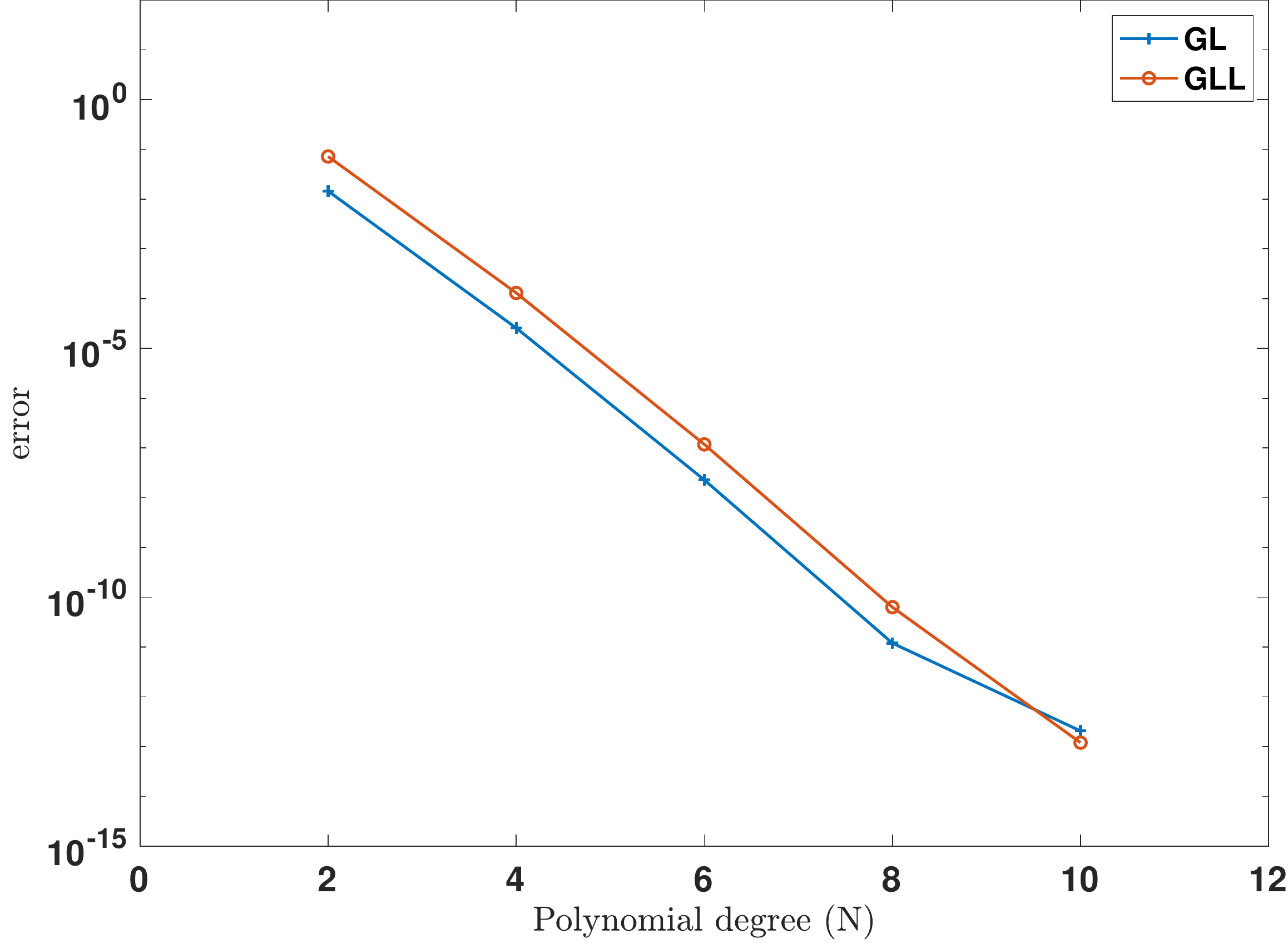}}  
    \caption{ Spectral convergence rate at $ t= 100 $ $s$.}
    \label{fig:spec_rate}
\end{figure}
\subsubsection{Dynamic rupture in 1D}
We will now consider an idealized dynamic earthquake rupture in 1D.
 The  domain is $0 \le x \le L = 60$ km, with homogeneous material properties,  $c_s = 3464$ m/s,  $\rho = 2670$ $\mathrm{kg/m^3}$, and  ${\mu} = \rho c_s^2 = 32.0381$ GPa.
 There is a fault at the middle of the domain, $x = 30$ km, with finite frictional strength. The two elastic solids separated by the fault are held together by a finite but high level frictional resistance.
The  nonlinear friction coefficient is prescribed by the slip-weakening friction law 
\begin{equation}\label{eq:slip-weakening}
\begin{split}
& f\left( S\right)  = \left \{
\begin{array}{rl}
 f_s -\left(f_s-f_d\right)\frac{ {S}}{d_c},  & \text{if}  \quad {S} \le d_c,\\
f_d, \quad {}  \quad {}& \text{if} \quad {S} \ge d_c,
\end{array} \right.
\end{split}
\end{equation}
where $ f_s$ and $ f_d$ are the static and dynamic friction coefficients, $d_c$ is the critical displacement and  the slip $S$ evolves according to
\begin{align}
\frac{dS}{dt} = V,
\end{align}
where $V = |\llbracket {v} \rrbracket|$ is the slip-rate. 
We introduce the peak frictional strength on the fault $\tau_p =  f_s \sigma_n$ and the residual  frictional strength on the fault $\tau_r =  f_d \sigma_n$, where $\sigma_n > 0$ is the compressive normal stress.
By \eqref{eq:slip-weakening}, as soon as the load on the fault exceeds the peak strength $\tau_p$, the fault will begin to slip and the strength on the fault will weaken linearly with slip $S$,  until slip  reaches the critical displacement  $S = d_c$. When the fault is fully weakened  the strength on the fault takes the value of the residual strength $\tau_r$. For the this simple 1D model, there is no mechanism to arrest ruptures. So once the fault nucleates it will slip forever.
 
 The parameter of friction are given in Table \ref{tab:friction} below, see also \cite{Harris2018}.
\begin{table}[h!]
\centering
\begin{tabular}{c |  c | c | c | c  }
 $f_s$ & $f_d$&  $d_c$[m] & $\sigma_n$[MPa] &  $\tau_0$[MPa]  \\
\hline
 0.677& 0.525& 0.4  & 120 & 81.6\\
\hline
\end{tabular}
\caption{Friction parameters.}
\label{tab:friction}
\end{table}
Note that $\tau_0 = 81.6 ~\ \text{MPa}$ is the initial load, and
$\tau_p =  f_s \sigma_n = 81.24~\ \text{MPa}$ and $\tau_r =  f_d \sigma_n = 63~\ \text{MPa}$ .
By the choice of the parameters, in Table \ref{tab:friction}, at the initial time the load will already exceed the peak strength $\tau_0 > \tau_p$. The initiation of rupture will be instantaneous and explosive.

We discretize the domain $[0, L]$ into 400 DG elements, and consider degree $N = 3$ polynomial approximation on GL nodes. Note that the effective grid spacing is $h = \Delta{x}/(N+1) = L/1600$.
As we have noted, the interface at $x = 30$ km is governed by the slip-weakening friction law, with the parameters given in Table \ref{tab:friction}. Except the interface at $x = 30$ km, every other DG interfaces are locked with infinite frictional strength, $\alpha \to \infty$. We use the time-step \eqref{eq:explicit_time_step}, and run the simulation for $t = 8$ s. 

In order to make a comparison, we perform numerical simulations with a SBP finite difference scheme, with the uniform grid spacing, $h = L/1600$. The SBP operator is 6th order accurate in the interior with 3rd order accurate boundary closure, yielding a 4th order accurate scheme globally.  As opposed to the DG method, where every inter-element boundary is a frictional interface,  in the finite difference scheme friction is only present at the fault $x = 30$ km.

In Figure \ref{fig:dynamic_rupture_fault}, we display the evolution of the slip-rate $V$, the shear stress $\tau$ on the fault and the fault slip $S$.
Note that the nucleation stage is explosive. That is, the shear stress weakens exponentially and the slip-rate increases exponentially. The fault accelerates until the fault is fully weakened, $\tau = \tau_r$ and the slip-rate reaches a constant value $V \sim 4$ m/s.
The fault continues slipping,  at a constant slip-rate $V \sim 4$ m/s, until the simulation is terminated.  
We also note that the results of both schemes are very similarly. In particular, the final slip predicted by the two algorithms are identical.
\begin{figure}[h!]
\begin{subfigure}
    \centering
       \stackunder[5pt]{\includegraphics[width=0.475\textwidth]{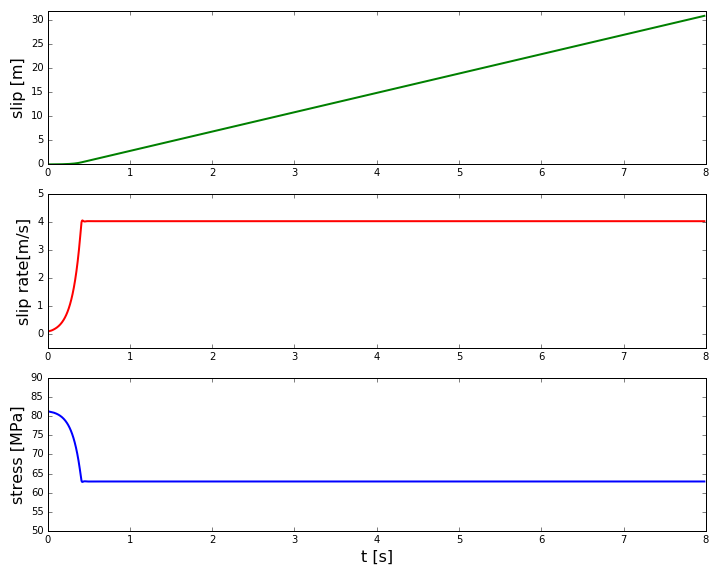}}{DG method.}
        \stackunder[5pt]{\includegraphics[width=0.475\textwidth]{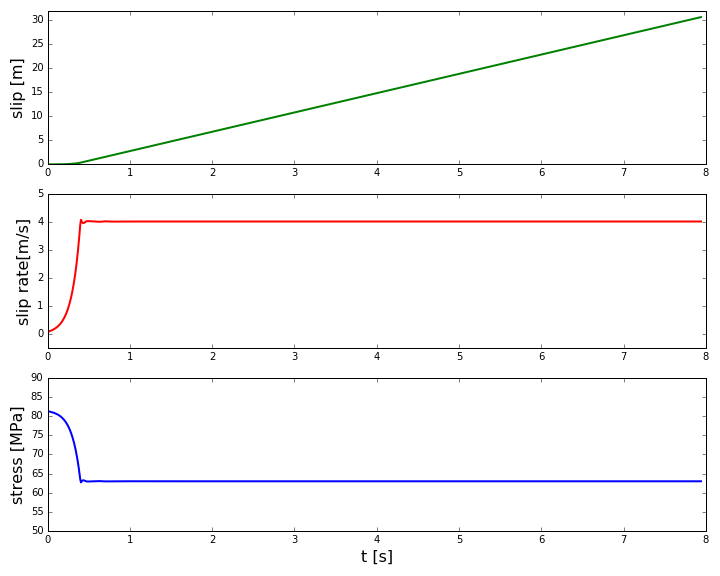}}{SBP method.}
     \end{subfigure}
    \caption{Evolution of the slip-rate $V$, the shear stress $\tau$ on the fault and the fault slip $S$.}
    \label{fig:dynamic_rupture_fault}
\end{figure}

Snapshots of the particle velocity and the stress are shown in Figure \ref{fig:snapshots_dynamic_rupture}. The stress is continuous across the fault interface, but the stress drop $\Delta{\tau} = \tau_0-\tau_r$ propagates from the fault into the adjacent elastic solids. 
The particle velocity is discontinuous across the interface. The discontinuity is the measure of the slip-rate $V$, and it is carried by outgoing wave radiations into the elastic solids.

\begin{figure}[h!]
\begin{subfigure}
    \centering
       \stackunder[5pt]{\includegraphics[width=0.33\textwidth]{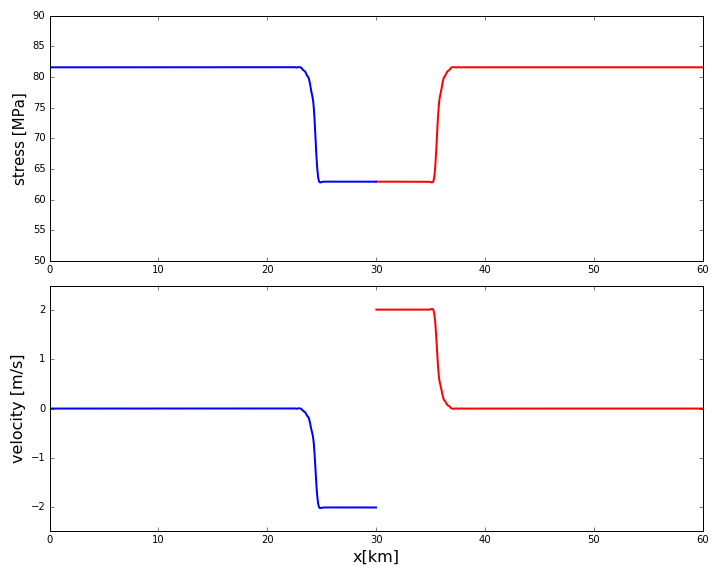}}{$t = 2$ s.}
        \stackunder[5pt]{\includegraphics[width=0.33\textwidth]{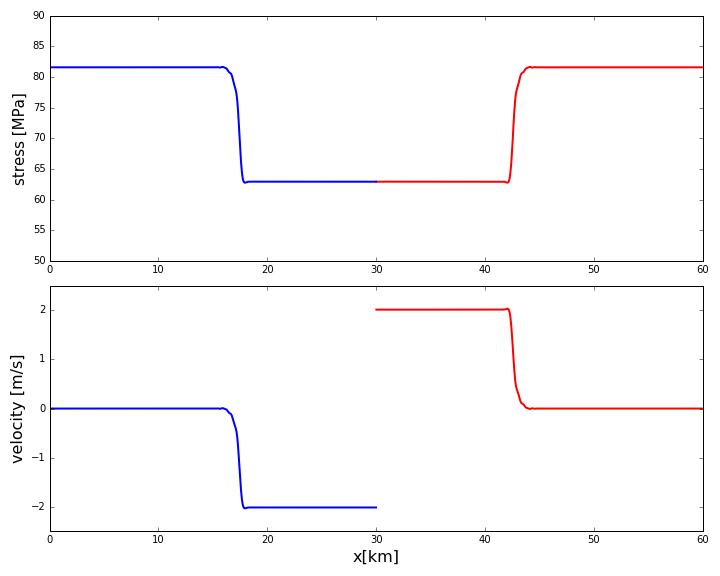}}{$t = 4$ s.}
        \stackunder[5pt]{\includegraphics[width=0.33\textwidth]{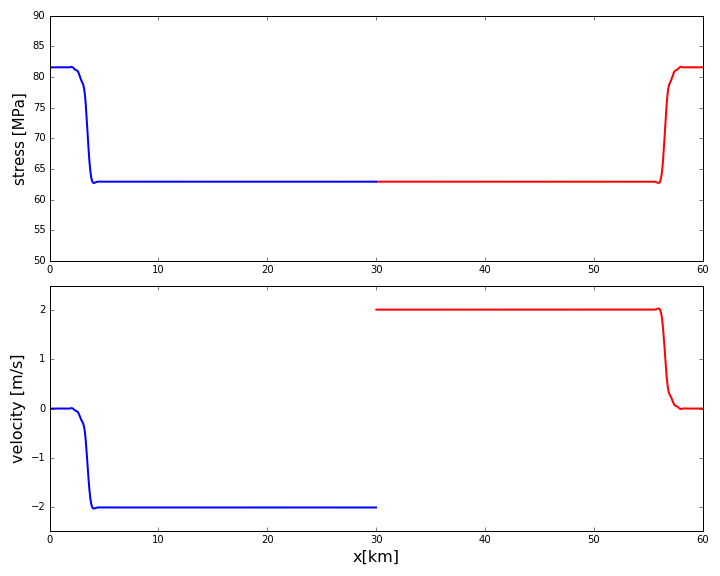}}{$t = 8$ s.}
     \end{subfigure}
    \caption{Snapshots of the velocity and the stress fields using the DG method.}
    \label{fig:snapshots_dynamic_rupture}
\end{figure}

In Figure \ref{fig:snapshots_dynamic_rupture_sbp_vs_dg}, we compare the wave fields for the DG method and SBP method at $t = 8$ s.
Note that the two solutions are similar. However, for the SBP method there are high frequency oscillations trailing the discontinuities.
In 2D and 3D  the oscillations will not only be present in the medium, it will also be present on the fault surface because of large (temporal and spatial) gradients in the slip-rates and stress fields.

\begin{figure}[h!]
\begin{subfigure}
    \centering
         \stackunder[5pt]{\includegraphics[width=0.475\textwidth]{ArXiv/dg_body_o3_t8s.png}}{DG method.}
        \stackunder[5pt]{\includegraphics[width=0.475\textwidth]{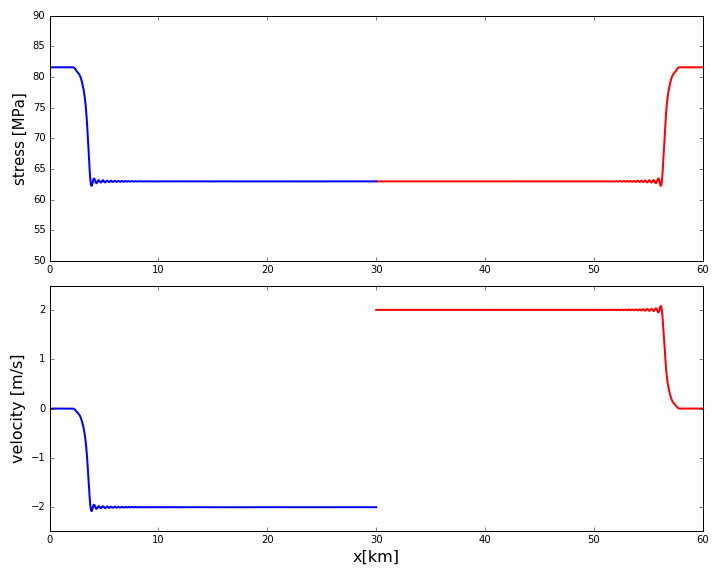}}{SBP method.}
     \end{subfigure}
    \caption{A comparison of the velocity and the stress fields  at $t = 8$ s.}
    \label{fig:snapshots_dynamic_rupture_sbp_vs_dg}
\end{figure}

\subsection{Two space dimensions}
Here, we perform numerical experiments in 2D and make comparisons with the Rusanov flux DG method \cite{DumbserPeshkovRomenski}. 
We present simulations on a curvilinear mesh, and propagate dynamic  earthquake ruptures on a dynamically adaptive Cartesian mesh.

\subsubsection{Comparison with the Rusanov flux}
Consider the 2D rectangular Poisson solid, with  $0 \le x \le 1$ km, $0 \le y \le 1$ km, and $\rho = 1000$ kg/m$^3$, $\lambda = 1 $ GPa,  $\mu = 1 $ GPa, where $\lambda, \mu$ are the first and second Lam\'e parameters.  At the top boundary $y = 0$ we set a free-surface boundary condition, while at all other boundaries we set the incoming characteristic to zero. The setup models a 2D half-space problem with the free-surface boundary condition at the surface $ y=0$, $\sigma_{xy}(x, 0, t) = 0$, $\sigma_{yy}(x, 0, t) = 0$. We initialize the particle velocity with a Gaussian perturbation 
\begin{align*}
v_x(x,y,0) = v_y(x,y,0) = e^{-5\frac{\left((x-0.5)^2 + (y-0.5)^2\right)}{0.01}},
\end{align*}
centered at $x =y = 0.5 $ km, and the stress fields are initially set to zero, $\sigma_{xx}(x, y, 0) = 0$, $\sigma_{yy}(x, y, 0) = 0$,  $\sigma_{xy}(x, y, 0) = 0$.

We discretize the medium with $81\times 81$ elements in both directions with a polynomial approximation of degree $N = 4$. We use GL quadrature nodes and advance the solutions until $ t= 10$ s. The snapshots of the solutions at $t = 0.5, 1.0, 2.77$ s are shown in Figure \ref{fig:Rusanov} b). In order to make a comparison we have run the simulations for the same setup  and material parameters using the Rusanov flux, see Figure \ref{fig:Rusanov} a). Note that initially (at $ t = 0.5, 1.0$ s) the solutions are visually comparable for both the Rusanov flux and our physically motivated flux. However, as time passes the numerical solution for the Rusanov flux generates instabilities from the boundaries.  At $ = 2.77$ s, these instabilities have eventually  corrupted the solution everywhere in the simulation domain.

We have also performed numerous numerical experiments by varying the velocity ratio $\gamma = c_p/c_s$, with $c_p = \sqrt{\left(2\mu + \lambda\right)/\rho}$, $c_s = \sqrt{\mu/\rho}$. The numerical instability for the Rusanov flux appears to be more severe when $\gamma \gg 2$ (i.e. the solutions blow up much earlier).  For the physically motivated numerical flux the solution is stable for all velocity ratios  $\gamma$. This is consistent with the theory, since the numerical method is provably stable.

\begin{figure}[h!]
    \centering
     \includegraphics[width=0.99\textwidth]{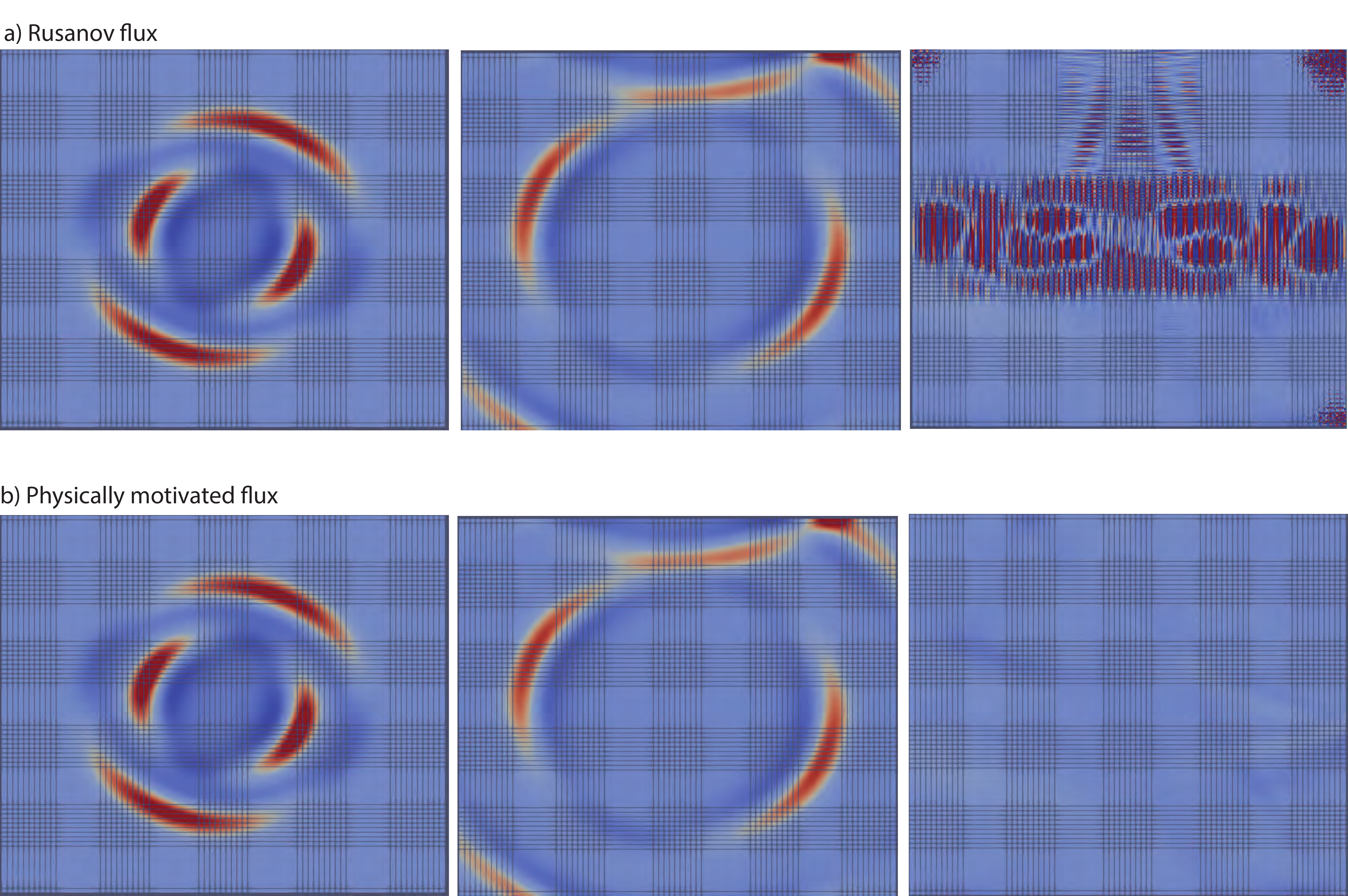}
    \caption{ A 2D example. a) The Rusanov flux showing numerical instabilities from boundaries. b) The physically motivated  flux showing stable solutions. The snapshots are at $t = 0.5, 1.0, 2.77$ s . All simulation and physical parameters are identical.}
    \label{fig:Rusanov}
\end{figure}

\subsubsection{Accuracy of Rayleigh surface waves}
Surface waves are propagating waves whose amplitudes are largest on the boundaries but decay exponentially into the domain. Here we demonstrate the effectiveness of the method  for computing surface waves in an elastic medium.  The 2D   elastic wave equation  in a half-plane $-\infty < x  <\infty $, $0 \le y < \infty $, with the free-surface boundary condition at $y = 0$, $\sigma_{xy}(x, 0, t) = 0$, $\sigma_{yy}(x, 0, t) = 0$, can  support surface waves. We  consider specifically  Rayleigh surface waves, see \cite{Rayleigh1885, DuruKreissandMattsson2015, KreissPeterssonYstrom2004}. For a constant coefficients $x$-periodic problem with the free-surface boundary condition at $y = 0$, the displacement field satisfies the Rayleigh wave solution
\begin{equation}\label{eq:surface}
\begin{split}
&\begin{pmatrix}
u_x(x,y,t)\\
u_y(x,y,t)\\
\end{pmatrix}
 = e^{-\omega  \sqrt{1-\tilde{\xi}^2} y}
\begin{pmatrix}
\cos{\left(\omega\left(x + c_r t\right)\right)}\\
\sqrt{1-\tilde{\xi}^2}\sin{\left(\omega\left(x + c_r t\right)\right)}\\
\end{pmatrix}
\\
&+
\left(\frac{\tilde{\xi}^2}{2}-1\right)e^{-\omega  \sqrt{1-\tilde{\xi}^2\mu/\left(2\mu + \lambda\right)}y}
\begin{pmatrix}
\cos{\left(\omega\left(x + c_r t\right)\right)}\\
\sin{\left(\omega\left(x + c_r t\right)\right)}/\sqrt{1-\tilde{\xi}^2\mu/\left(2\mu + \lambda\right)}\\
\end{pmatrix}.
\end{split}
\end{equation}
Here $\omega > 0,$  $c_r = \tilde{\xi}\sqrt{\mu}$ is the Rayleigh phase velocity, and $\tilde{\xi}$ satisfies the Rayleigh dispersion relation
\begin{equation}
\begin{split}
\sqrt{1-\tilde{\xi}^2}\sqrt{1-\frac{\tilde{\xi}^2\mu}{\left(2\mu + \lambda\right)}} - \left(\frac{\tilde{\xi}^2}{2}-1\right)^2 = 0.
\end{split}
\end{equation}
Note that  for all $\mu > 0$ and  $\lambda \ge 0$ we  must have $0.763 < \tilde{\xi}^2< 0.913$. Thus, the Rayleigh surface wave propagates in the $x$-direction and decays exponentially in the $y$-direction. The velocity field can be extracted from \eqref{eq:surface}, by taking the time derivative of the displacement field, giving
\begin{equation}
\begin{split}
{v}_x(x,y,t)  = \frac{\partial {u}_x(x,y,t) }{\partial t}, \quad {v}_y(x,y,t)  = \frac{\partial {u}_y(x,y,t) }{\partial t}.
\end{split}
\end{equation}
The stress field can be obtain from \eqref{eq:surface}, by combining the spatial gradients of the displacement field with the stiffness tensor of elastic material,  as prescribed by Hooke's law. We have
\begin{equation}
\begin{split}
\sigma_{xx}(x,y,t)  &= (2\mu + \lambda)\frac{\partial {u}_x(x,y,t) }{\partial x} + \lambda \frac{\partial {u}_y(x,y,t) }{\partial y}, 
\quad 
\sigma_{yy}(x,y,t)  = \lambda\frac{\partial {u}_x(x,y,t) }{\partial x} + (2\mu + \lambda) \frac{\partial {u}_y(x,y,t) }{\partial y},\\
\quad
\sigma_{xy}(x,y,t)  &= \mu\left(\frac{\partial {u}_x(x,y,t) }{\partial y} +  \frac{\partial {u}_y(x,y,t) }{\partial x}\right).
\end{split}
\end{equation}

We consider the $x$-periodic rectangular domain, $0\le x \le 1$ km, $0 \le y \le 10$ km, with $\omega = 2 \pi$. Note  that in the $x$-direction the solution is $1$-periodic,  at $y = 0$ we have the free-surface boundary condition and  at $ y = 10$ km we prescribe a Dirichlet condition for the velocity field. 

We use  $N = 4$ degree polynomial approximation on GL and GLL nodes separately, and evaluate numerical accuracy, on a sequence of uniformly refined meshes. We consider the relative $L_2$-norm error for the particle velocity vector and the stress vector, separately.  First we consider the Poisson solid with $\rho = 1000$ kg/m$^3$, $\lambda = 1000 $ MPa,  $\mu = 1000 $ MPa, with $\lambda/\mu =1$.  The final time is $t = 1$ s. Numerical errors are at the final time $t = 1$ s are shown in Tables \ref{tab:RayleighSurfaceWaves_Velocity_1} and \ref{tab:RayleighSurfaceWaves_Stress_1} for the particle velocity and the stress field respectively. In the asymptotic regime the errors converge optimally (at the rate $N+1$).

\begin{table}[h!]
\centering
\begin{tabular}{c |  c | c | c | c  }
 $\Delta{x}$ & error(GLL)&  rate(GLL) & error(GL)&  rate(GL)  \\
\hline
1&3.1560e-01 & -- & {1.0210e-01} & {--} \\
0.5&  2.3600e-02 & 3.7395 & {3.7000e-03} & {4.7803}  \\
0.25&6.8522e-04 & 5.1077 & {9.8056e-05 } & {5.2436}  \\
0.125&2.2145e-05 & 4.9515 & {3.2188e-06} & {4.9290}  \\
0.0625&  6.9068e-07 &  5.0029 & {1.0016e-07} & {5.0062} \\
\hline
\end{tabular}
\caption{Relative  numerical errors of the particle velocity and convergence rate at $t = 1.0$ s with $\lambda/\mu = 1$.}
\label{tab:RayleighSurfaceWaves_Velocity_1}
\end{table}

\begin{table}[h!]
\centering
\begin{tabular}{c |  c | c | c | c  }
 $\Delta{x}$ & error(GLL)&  rate(GLL) & error(GL)&  rate(GL)  \\
\hline
1&2.7980e-01 & -- & {1.1370e-01} & {--} \\
0.5&1.8400e-02 & 3.9253 & {4.0000e-03} & {4.8344}  \\
0.25&8.8761e-03 & 4.3750 & {1.4461e-04 } & {4.7849}  \\
0.125&2.7906e-05 & 4.9913 & {4.4875e-06} & {5.0101}  \\
0.0625& 8.2938e-07 &  5.0724 &  {1.3488e-07} &  {5.0561} \\
\hline
\end{tabular}
\caption{Relative numerical errors of the stress field and convergence rate at $t = 1.0$ s with $\lambda/\mu = 1$. }
\label{tab:RayleighSurfaceWaves_Stress_1}
\end{table}

The analysis in \cite{DuruKreissandMattsson2015, KreissPeterssonYstrom2004}, shows that surface waves are very sensitive to numerical errors in almost incompressible elastic materials, that is when $\lambda/\mu \gg 1$. Higher order accurate  numerical schemes become essential for accurate and efficient numerical simulations. To investigate this, we consider $\lambda/\mu = 100$, where $\rho = 1000$ kg/m$^3$, $\lambda = 100000 $ MPa,  $\mu = 1000 $ MPa. Numerical errors at the final time $t = 1$ s are shown in Tables \ref{tab:RayleighSurfaceWaves_Velocity} and \ref{tab:RayleighSurfaceWaves_Stress} for the particle velocity and the stress field respectively. Note that for the particle velocity, the amplitude of the relative errors seems unaffected by the velocity ratio. For the stress field, the increase of the the velocity ratio from $\lambda/\mu = 1$ to $\lambda/\mu = 100$ leads to the increase of the relative error by a factor of 4. However, for both cases $\lambda/\mu = 1$ and $\lambda/\mu = 100$, the relative error converges optimally to zero in the asymptotic regime.
\begin{table}[h!]
\centering
\begin{tabular}{c |  c | c | c | c  }
%
 $\Delta{x}$ & error(GLL)&  rate(GLL) & error(GL)&  rate(GL)  \\
\hline
1&3.4170e-01 & -- & {1.0800e-01} & {--} \\
0.5&  2.1400e-02 & 3.9994 & {3.4000e-03} & {4.9698}  \\
0.25&6.9680e-04 & 4.9383 & {1.0752e-04 } & {5.0024}  \\
0.125&2.3167e-05 & 4.9104 & {3.4860e-06} & {4.9469}  \\
0.0625&7.2740e-07 & 4.9934 & {1.1024e-07} & {4.9828} \\
\hline
\end{tabular}
\caption{Relative numerical errors of the particle velocity and convergence rate at $t = 1.0$ s with $\lambda/\mu = 100$.}
\label{tab:RayleighSurfaceWaves_Velocity}
\end{table}
\begin{table}[h!]
\centering
\begin{tabular}{c |  c | c | c | c  }
%
 $\Delta{x}$ & error(GLL)&  rate(GLL) & error(GL)&  rate(GL)  \\
\hline
1&4.0050e-01 & -- & {1.7220e-01} & {--} \\
0.5&3.7900e-02 & 3.4034 & {1.0600e-02} & {4.0188}  \\
0.25&2.2000e-03 & 4.0728 & {5.2501e-04 } & {4.3387}  \\
0.125&9.2341e-05 & 4.6063 & {1.8313e-05} & {4.8414}  \\
0.0625&3.0247e-06 & 4.9321& {5.5659e-07} & {5.0401} \\
\hline
\end{tabular}
\caption{Relative numerical errors of the stress field and convergence rate at $t = 1.0$ s with $\lambda/\mu = 100$. }
\label{tab:RayleighSurfaceWaves_Stress}
\end{table}


\subsubsection{Non-planar topography}
Here, we demonstrate the potential of our method in modeling geometrically complex free surface topography.
Consider the 2D  isotropic elastic medium, with  $-10 \le x \le 10$ km, $0 \le y \le \widetilde{y}(x)$ km, and $ \widetilde{y}(x) = 10 + 0.1x + \sin\left(4\pi x/20 + 3.34\right) \cos\left(2\pi\left(x/20 -0.5\right) + 3.34\right)$. We use transfinite interpolation to propagate points on the boundaries into the domain,  resulting in a curvilinear mesh obeying the topography. To enable efficient numerical treatment, we map the mesh and the PDE to a regular Cartesian mesh. We discretize the transformed domain into a tensor-product of dG elements, and further discretize each element using GLL nodes.  Note that in the physical  space the elements are curved. See  Figure \ref{fig:mesh} for a graphical representation of the computational mesh. We consider a homogeneous  crustal rock material properties,with $\rho = 2700$ kg/m$^3$, $c_p = 6000 $ m/s, and $c_s = 3343 $ m/s, where $c_p$ is the p-wave speed and $c_s$ is the shear wave speed.  In the transformed Cartesian domain, however, the medium is heterogeneous and anisotropic. At the top boundary $ y = \widetilde{y}(x)$ we set a free-surface boundary condition, while at all other boundaries we set the incoming characteristic to zero. All boundary and inter-element conditions are implemented weakly, as discussed in the previous sections, by constructing appropriate data and penalizing the data against the incoming characterisitics, using physically motivated penalties. We initialize the normal stress ($\sigma_{xx}$, $\sigma_{yy}$) with a Gaussian perturbation centered at $x = 0$ km, $y = 6 $ km, while  the shear stress ($\sigma_{xy}$) and the particle velocity vector ($v_x$, $v_y$)  are initially set to zero. The initial condition  generates  pressure wave perturbation only. Snapshots of the absolute divergence: $\left|\frac{\partial v_x}{\partial x} + \frac{\partial v_y}{\partial y}\right |$, and the absolute curl: $\left|\frac{\partial v_y}{\partial x} - \frac{\partial v_x}{\partial y}\right |$, of the particle velocity vector are plotted in Figure \ref{fig:topography}, showing the evolution of the wave field and the interaction of waves with the non-planar topography. Note that initially, for $ t \le 2.3$ s, the absence of shear wave perturbation in the initial data implies that the curl of the velocity vector vanishes identically. However, as time progresses and the wave begin to interact with the free-surface topography, shear waves are generated due to mode conversions. This is evident in the curl of the velocity  field shown in Figure \ref{fig:topography} for $t \ge 0.62$ s. We have evolved the wave field for a sufficiently long time, $t \le 100$ s, without observing instabilities. Again, this is consistent with the expectations from theory, since the numerical method is provably stable.

\begin{figure}[h!]
\begin{subfigure}
    \centering
\stackunder[5pt]{\includegraphics[width=0.475\textwidth]{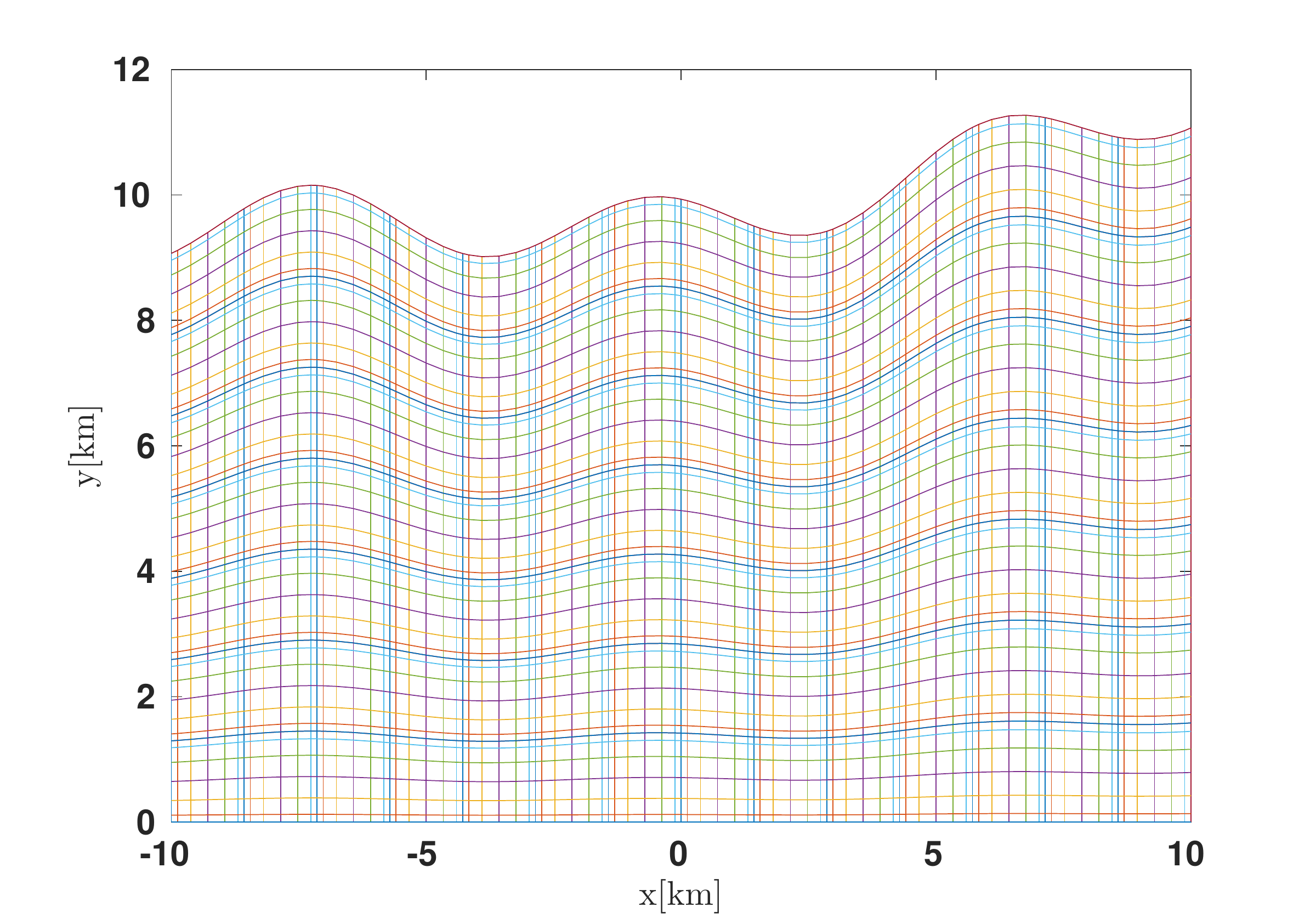}}{Boundary conforming curvilinear mesh.}%
\hspace{0.0cm}%
\stackunder[5pt]{\includegraphics[width=0.5\textwidth]{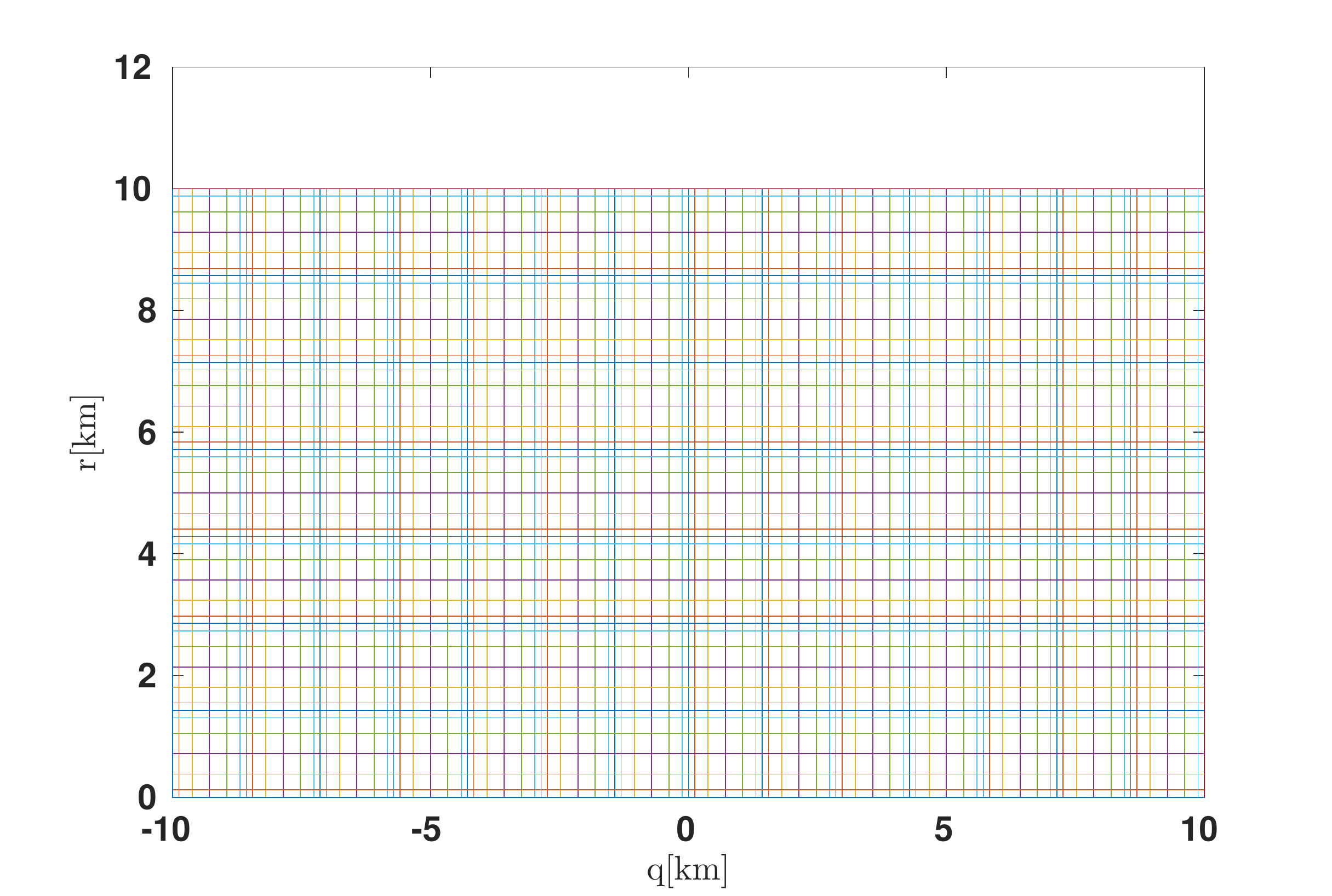}}{Transformed Cartesian mesh.}%
     \end{subfigure}
    \caption{ Computational mesh}
    \label{fig:mesh}
\end{figure}

\begin{figure}[h!]
    \centering
    \includegraphics[width=0.245\textwidth]{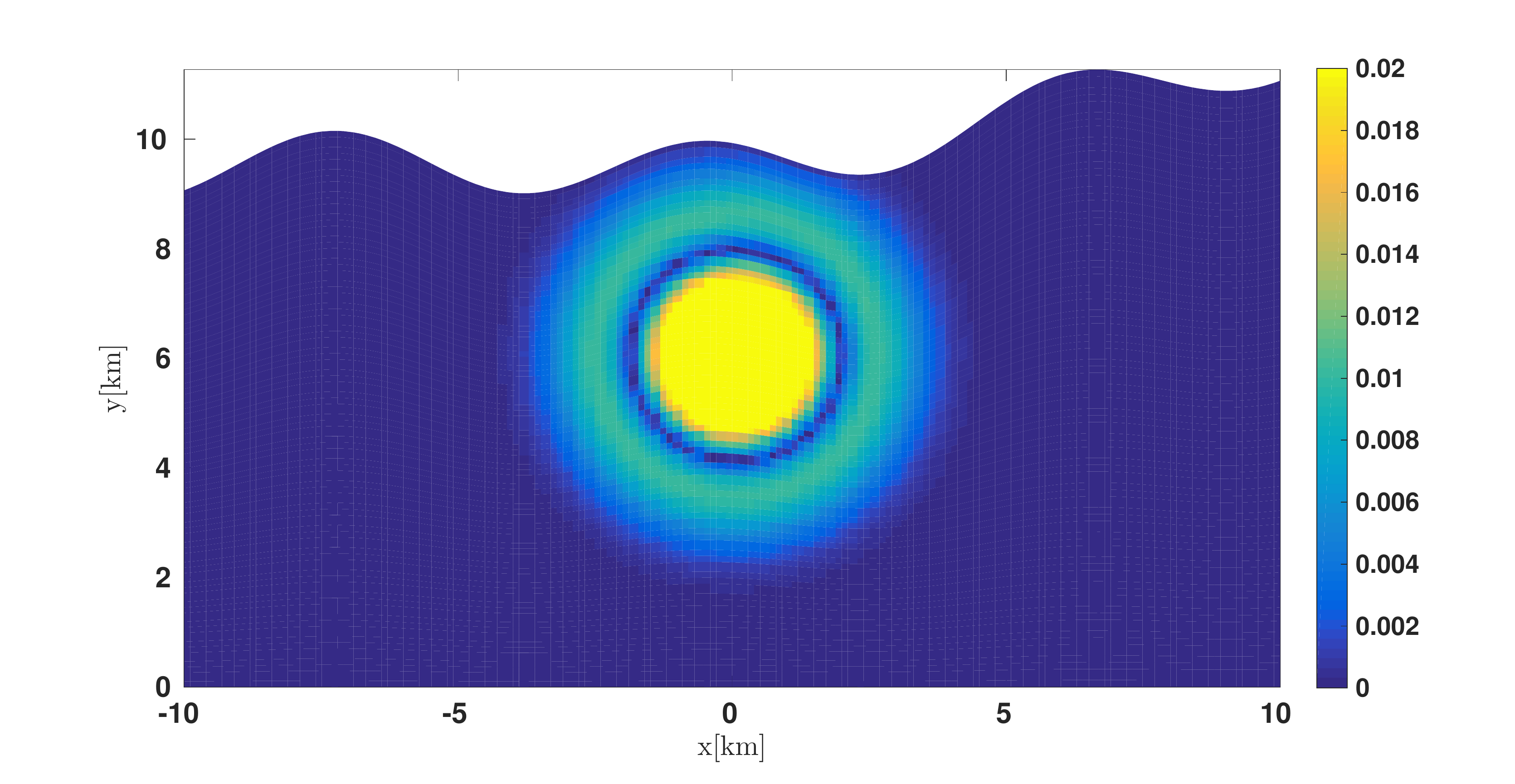}
     \includegraphics[width=0.245\textwidth]{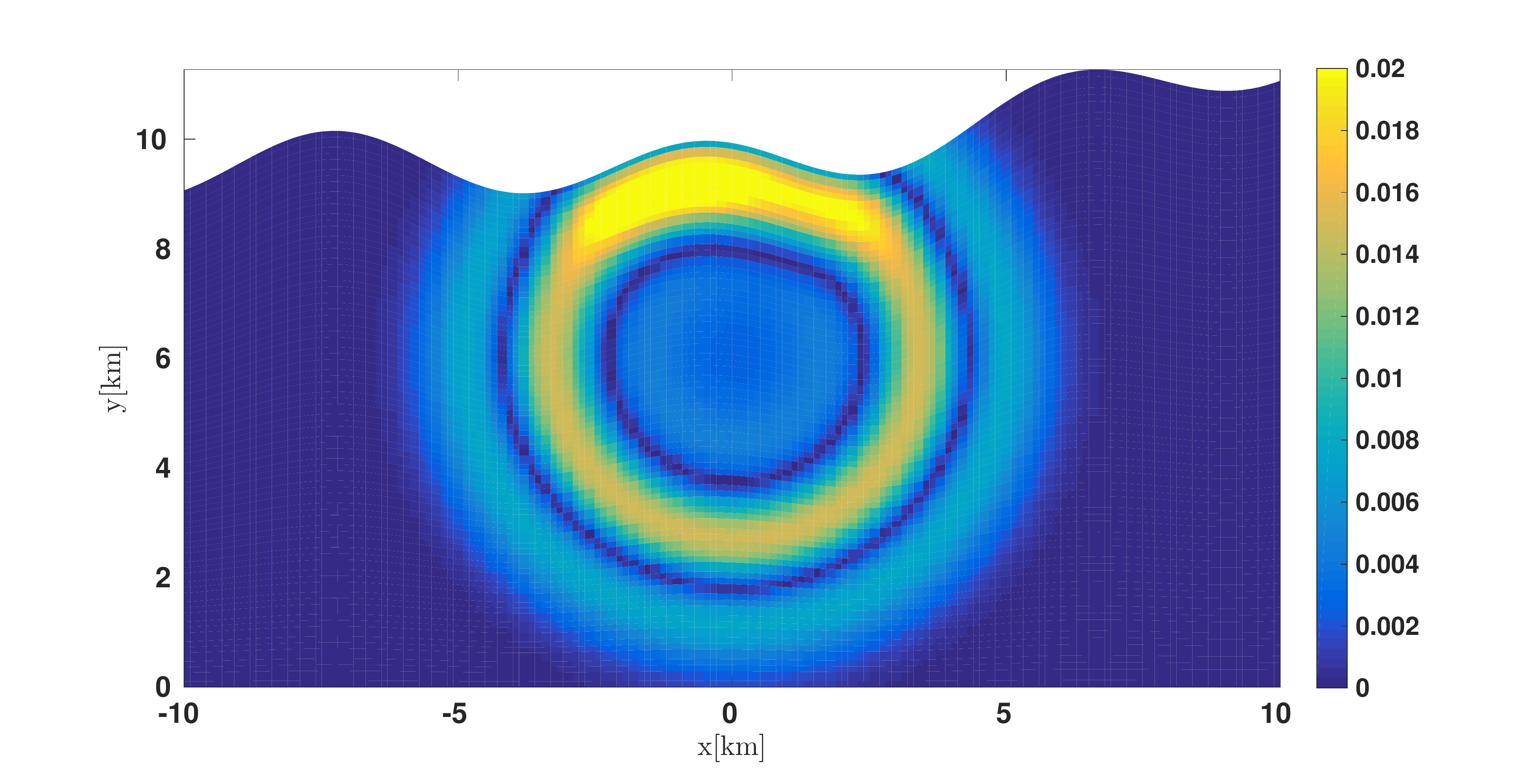}
     \includegraphics[width=0.245\textwidth]{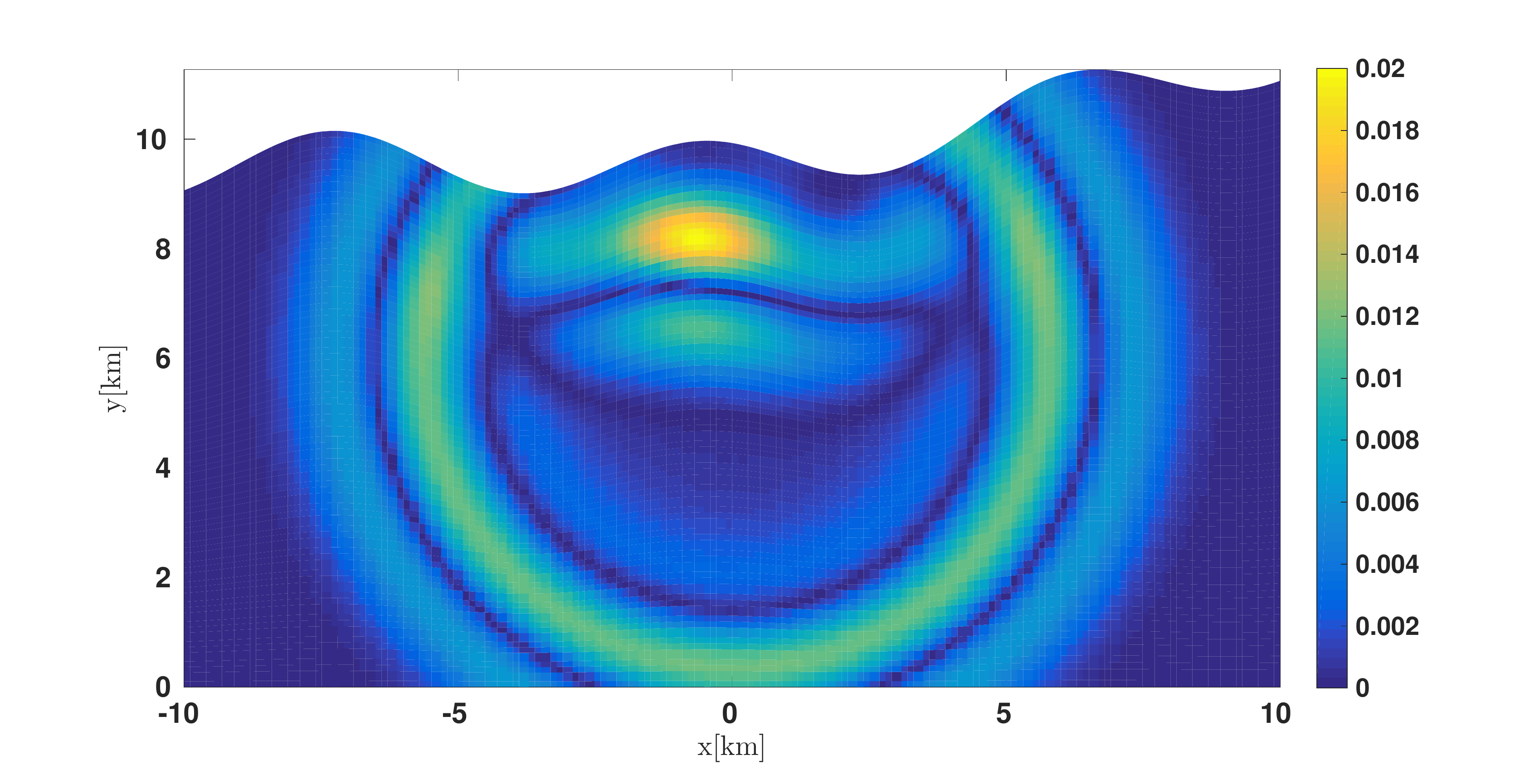}
     \includegraphics[width=0.245\textwidth]{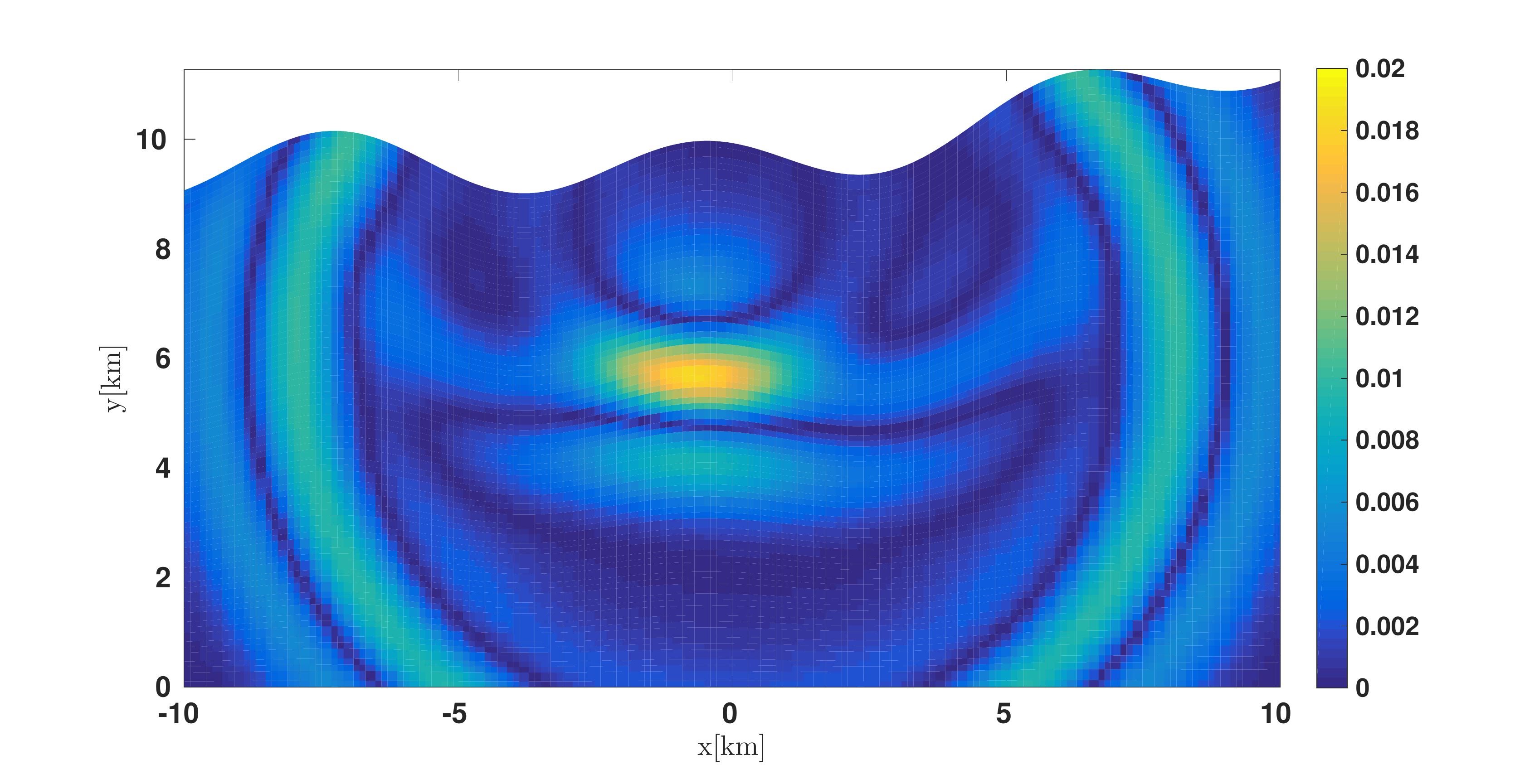}
      \includegraphics[width=0.245\textwidth]{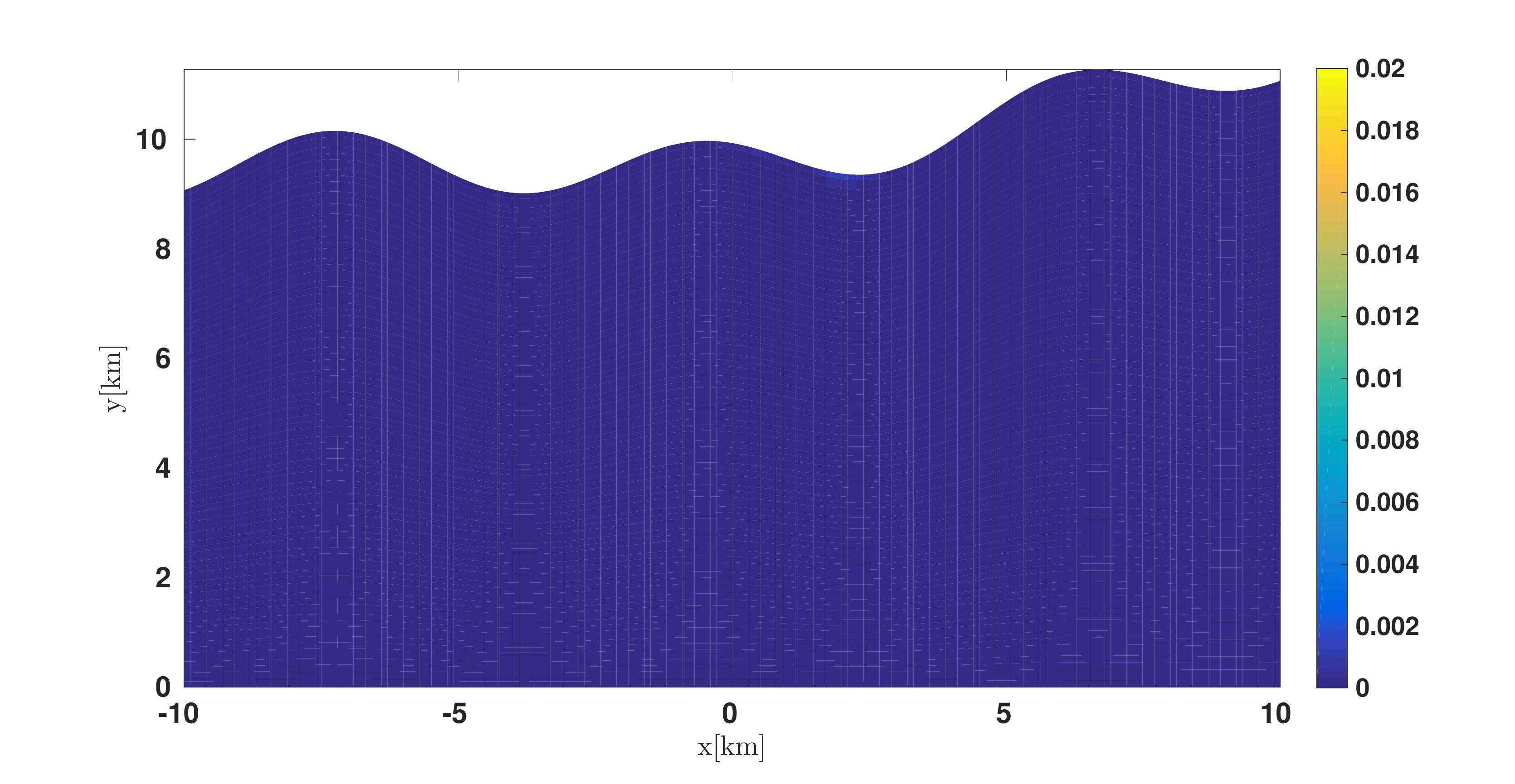}
     \includegraphics[width=0.245\textwidth]{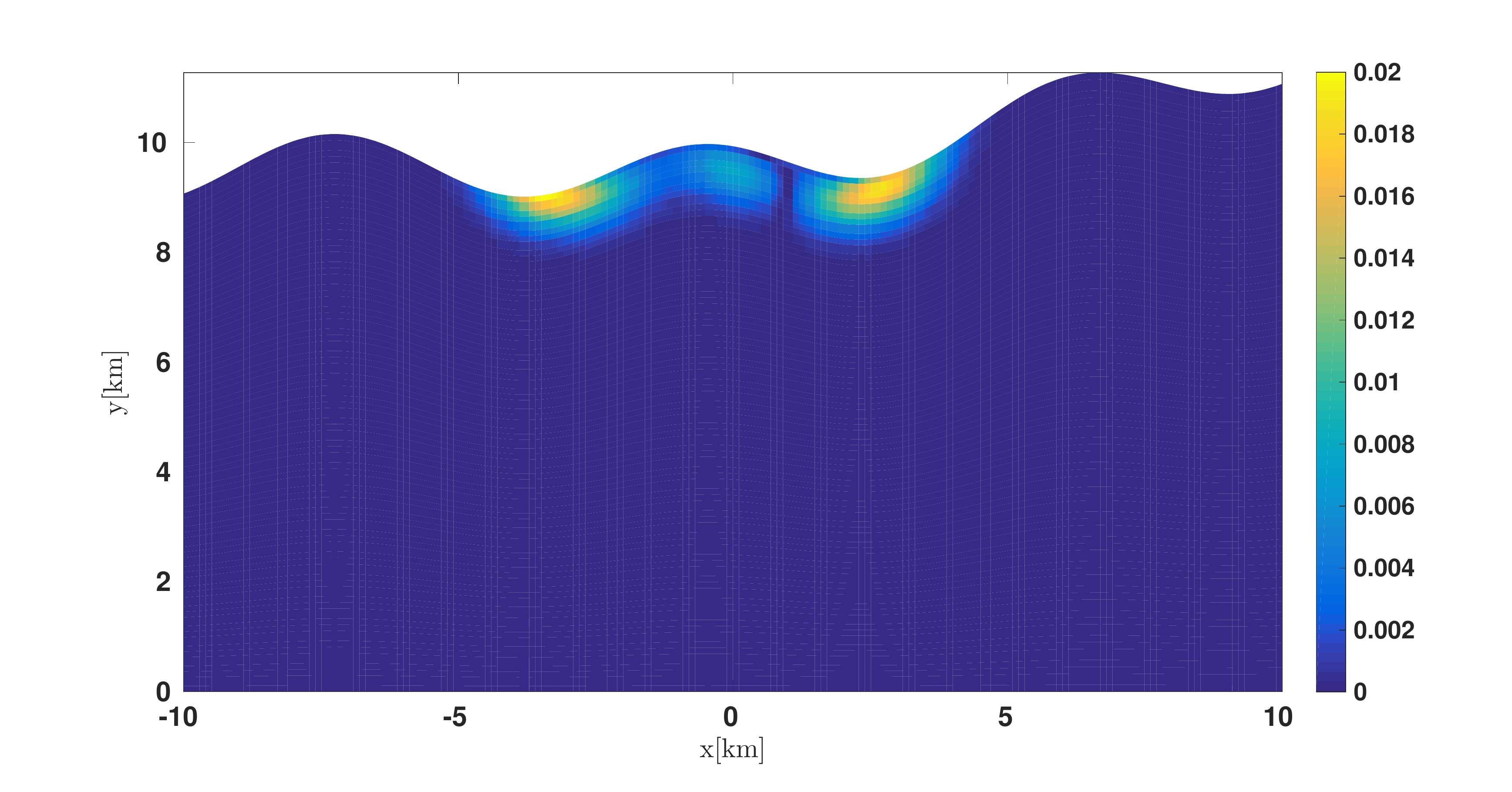}
     \includegraphics[width=0.245\textwidth]{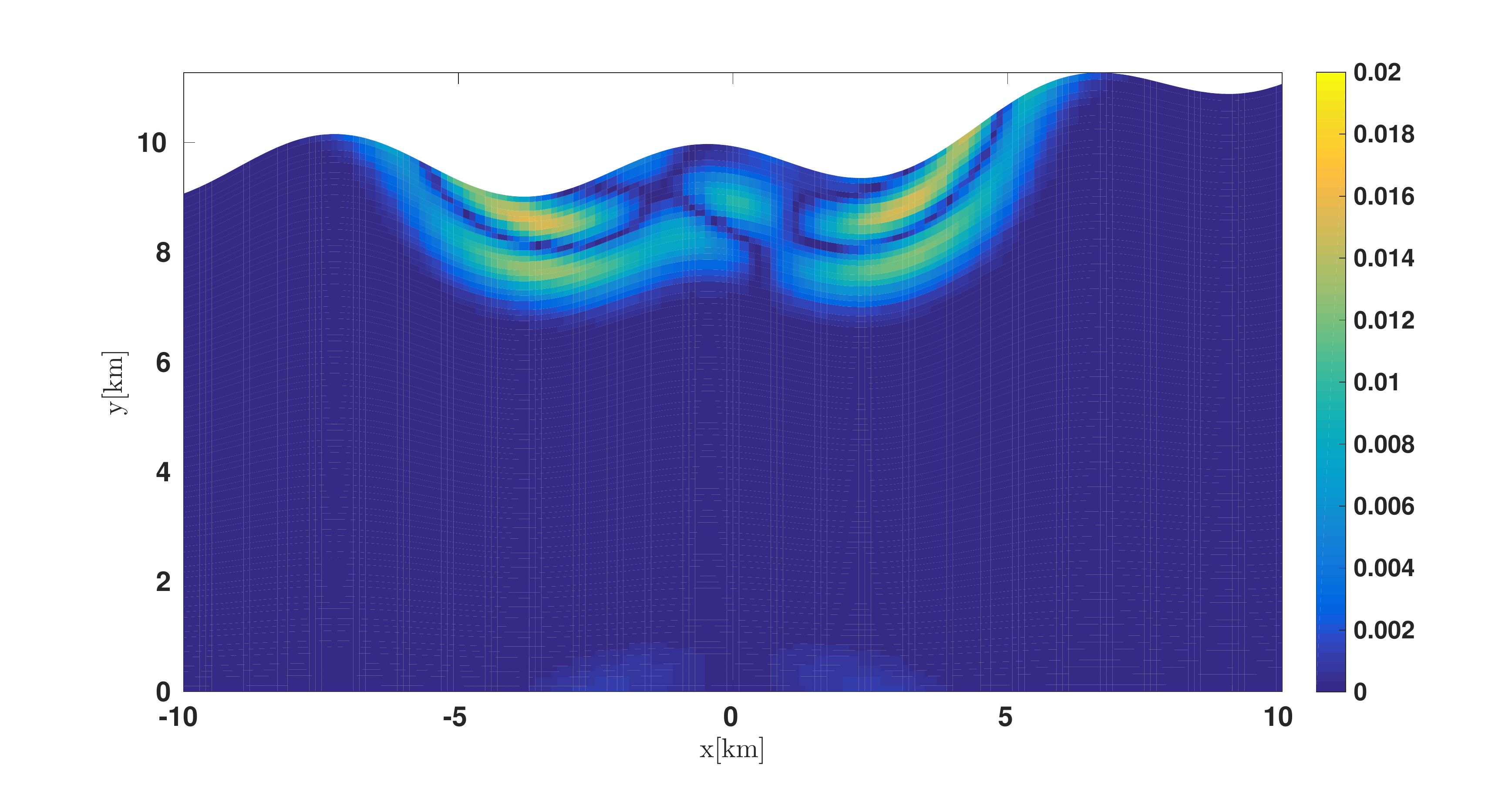}
     \includegraphics[width=0.245\textwidth]{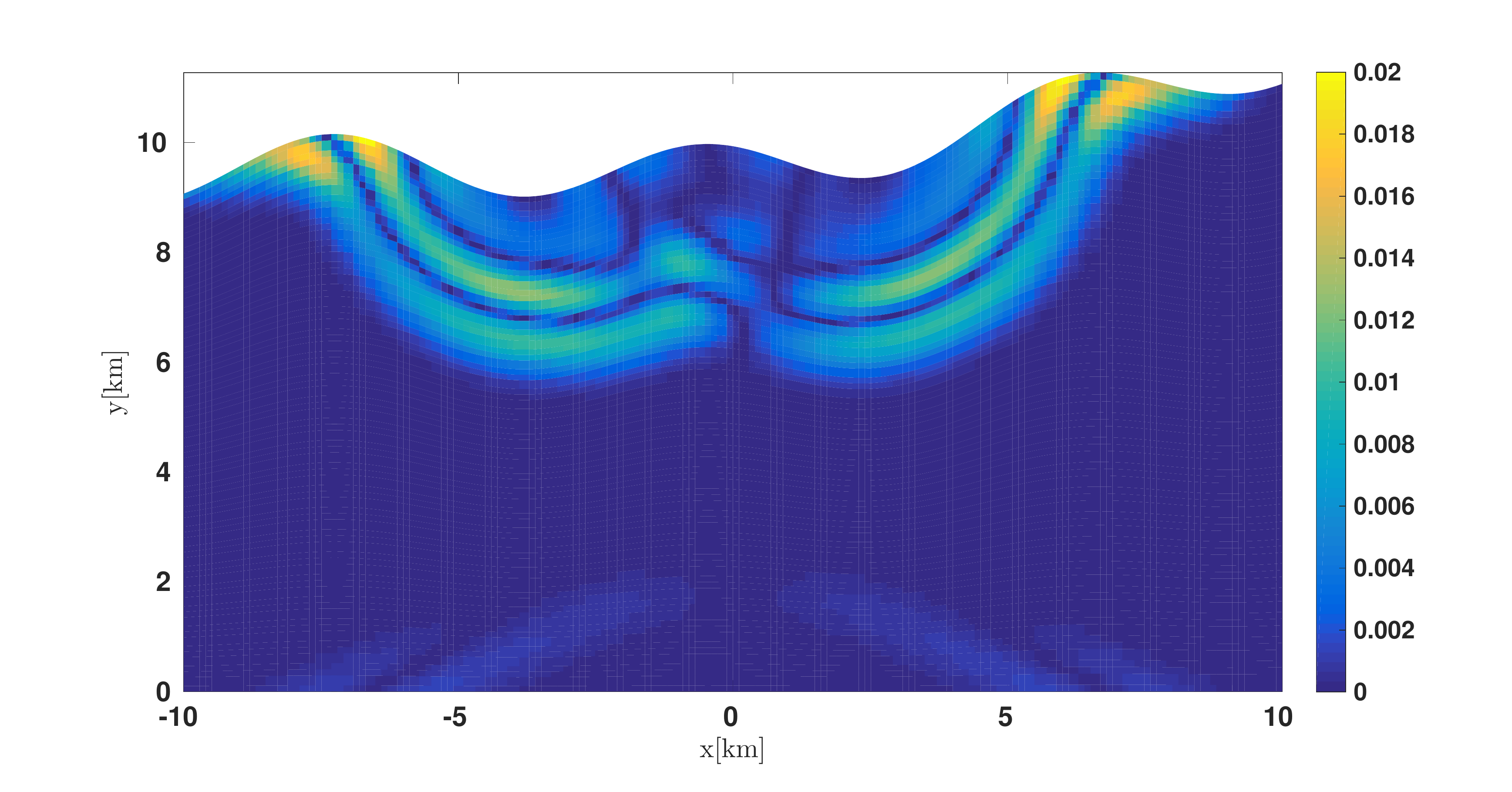}
    \caption{ Complex topography.  Snapshots of the wave field, from left to right, at $t = 0.23, 0.62, 1.01, 1.40$ $s$.  The top panel is the absolute divergence of the particle velocity vector and the lower panel  is the absolute curl of the particle velocity vector.}
    \label{fig:topography}
\end{figure}

Next, we will  perform numerical experiments to demonstrate the robustness and accuracy of the method.
\subsubsection{ Dynamic earthquake ruptures on a dynamically adaptive mesh}
We  will now  propagate  dynamic earthquake ruptures on a dynamically adaptive mesh. This numerical experiment is designed to demonstrate  the robustness of our method, and attempt an adaptive mesh refinement strategy. The domain span $(x, y) \in [0, 30~\text{km}] \times [0, 20~\text{km}]$.
Here, the fault is a vertical line subdividing the two isotropic elastic solids, at $x = 15~\text{km}$. The material properties of the  elastic solid are homogeneous
 $c_p=6000$ m/s, $c_s=3464$ m/s, $\rho=2670$ kg/m$^3$. The two elastic solids are held together by a slip-weakening friction law, \eqref{eq:slip-weakening}, 
with the friction parameters $f_s  = 0.677 $, $f_d  = 0.525 $, and $ d_c= 0.40$ m.
We consider initial uniform prestress distribution $ \sigma_{xy}^0 = 70$ MPa, $ \sigma_{yy}^0 = 0$ MPa, $ \sigma_{xx}^0= 120$ MPa.
At $t =0$ we discretize the domain uniformly with the element size $\Delta{x} = 30/21$ km, $\Delta{y} = 20/14$ km, and consider degree $ N = 5$ polynomial approximation on GL nodes.
The peak frictional strength on the fault is $\tau_p =  f_s \sigma_n = 81.24~\ \text{MPa}$.
We nucleate the fault at $y = 7.5$ km depth by over-stressing the  element containing $y = 7.5$ km, with $ \tau_0= 81.6$ MPa.

\begin{figure}[h!]
\begin{center}
	\includegraphics[width=0.3\columnwidth]{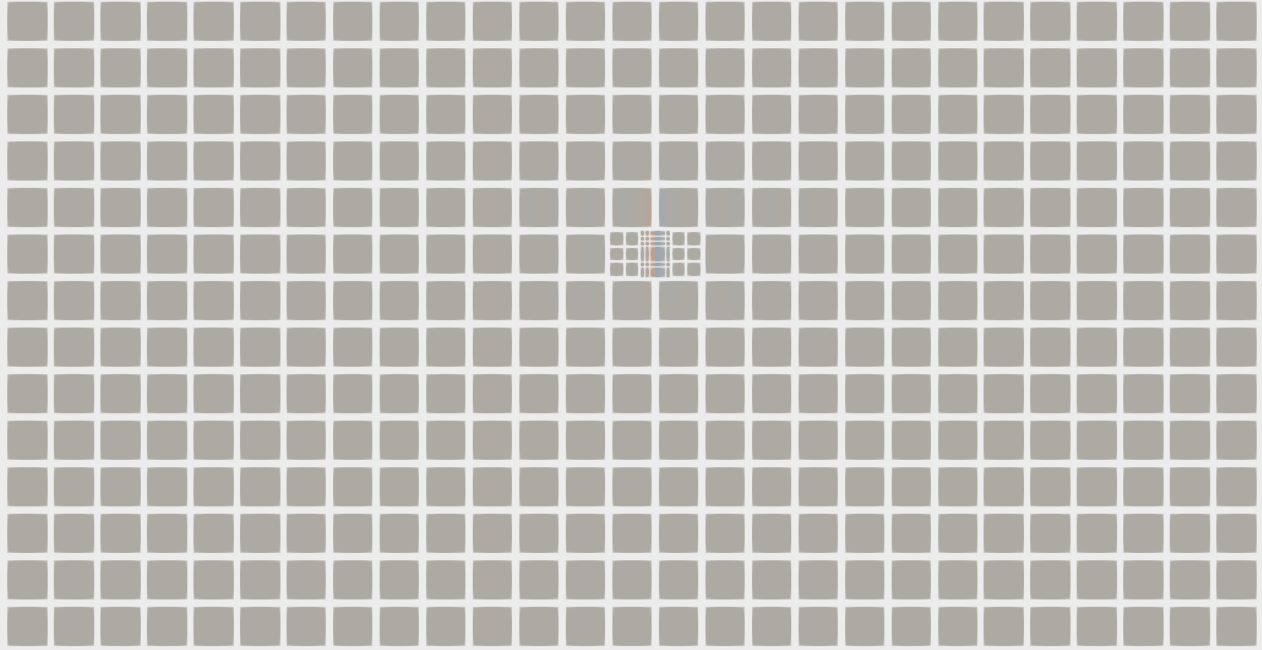} 
	\includegraphics[width=0.3\columnwidth]{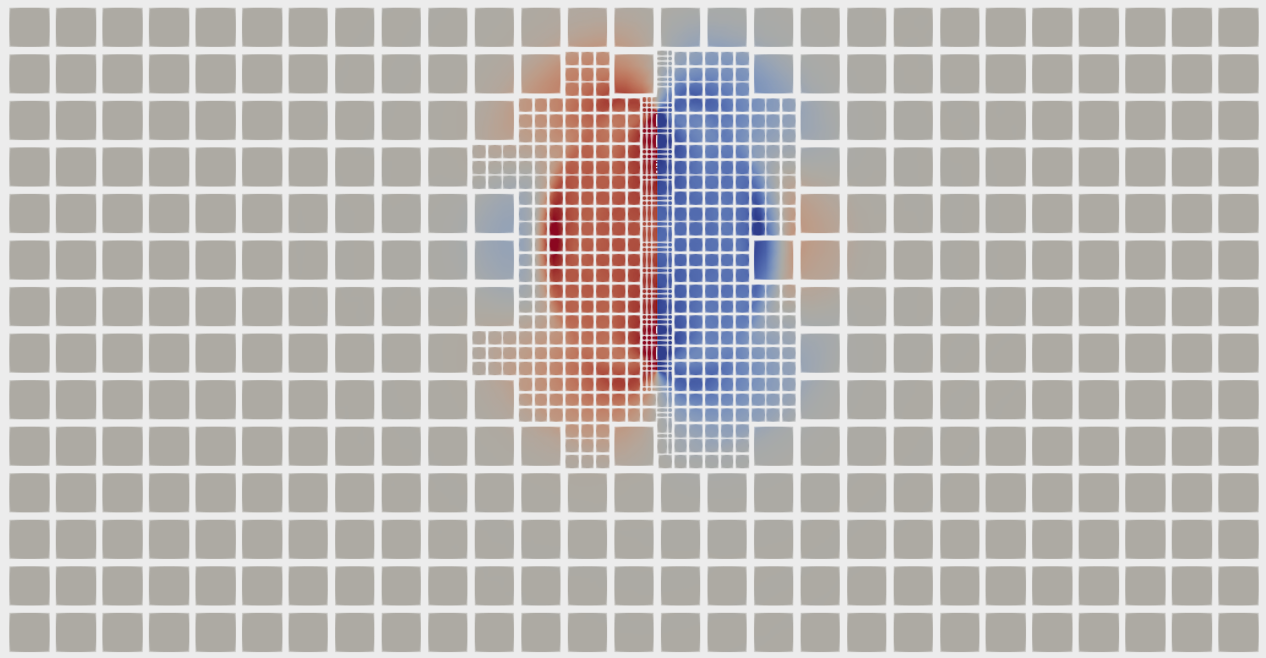} 
	\includegraphics[width=0.3\columnwidth]{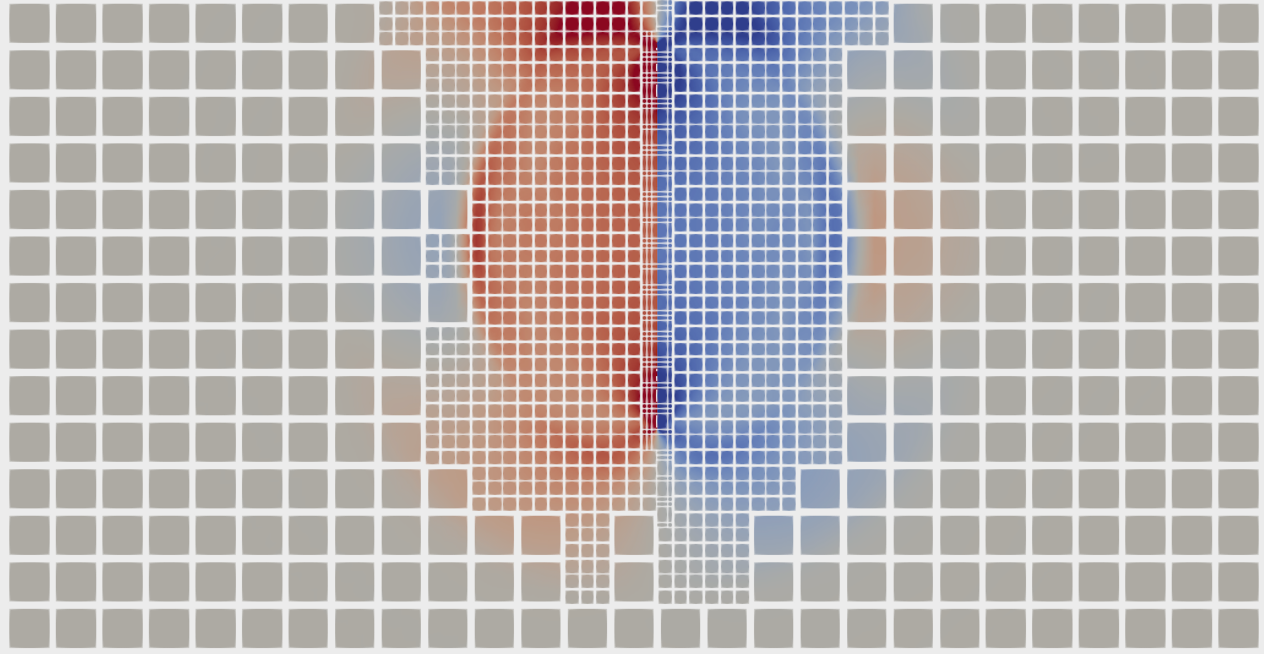} 
	\includegraphics[width=0.3\columnwidth]{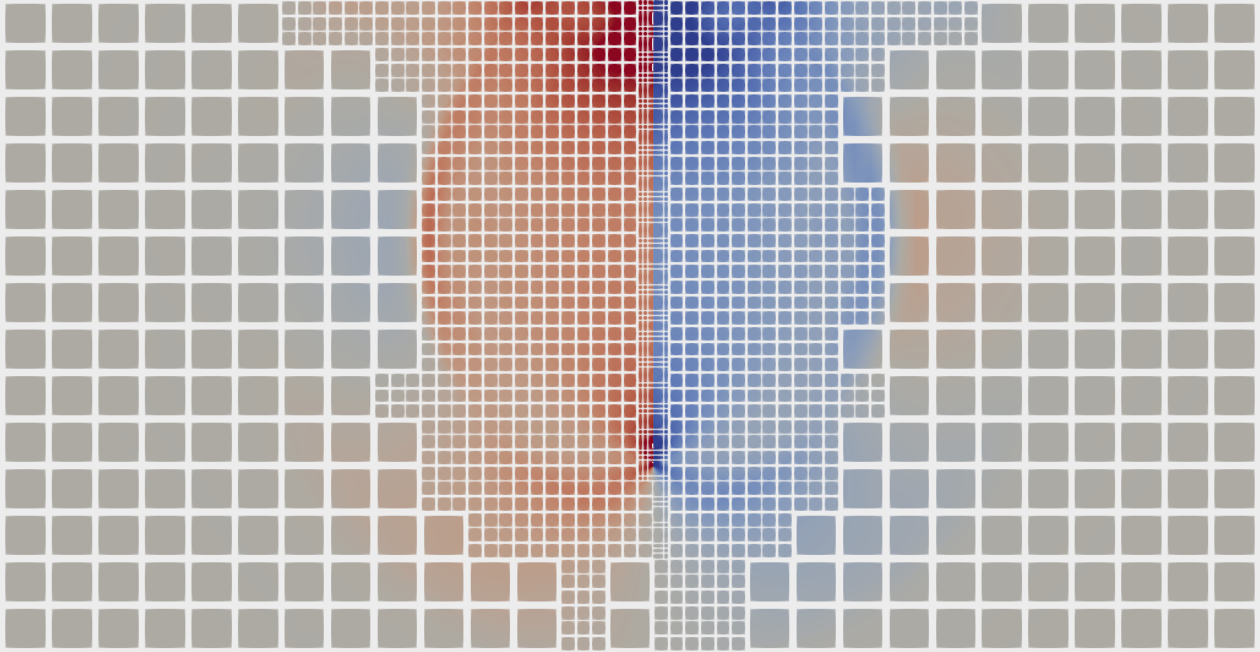} 
	\includegraphics[width=0.3\columnwidth]{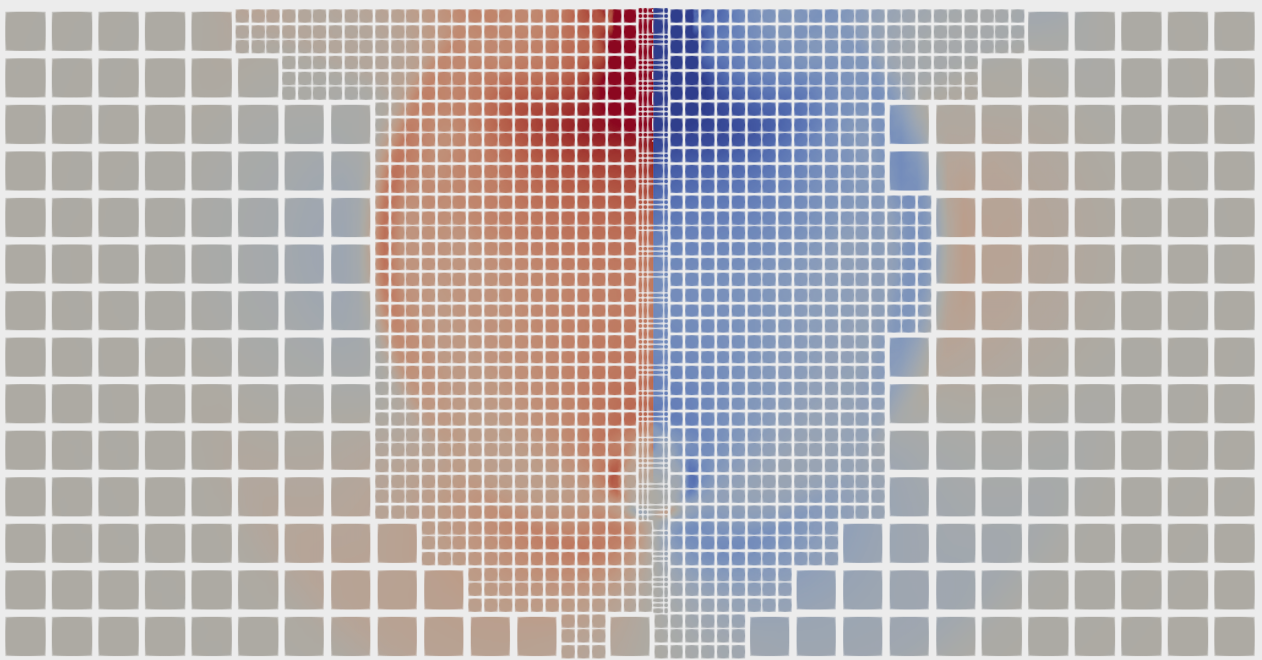} 
	\includegraphics[width=0.3\columnwidth]{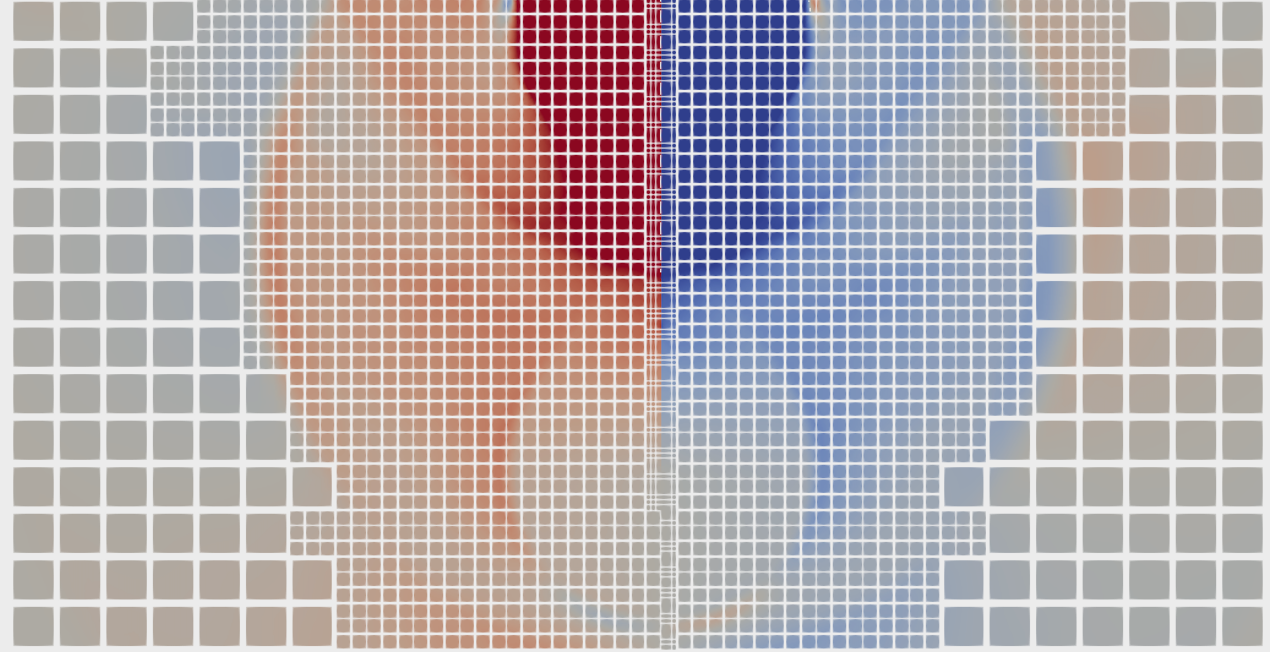} 
\end{center}
\caption{Earthquake ruptures on a dynamically adaptive mesh. The snapshots of the particle velocity $v_y$ are taken at $t = 0.1, 1, 2, 3, 4, 5$ s.}
\label{fig:adaptive_mesh_refinement}
 \end{figure}
For ruptures, refinement criteria is set by monitoring the slip rate $V$ on the fault, while the root means square of the particle velocity, $v = \sqrt{v_x^2 + v_y^2}$,  gives the mesh refinement indicator for the wave fields. That is, once the slip-rate at any point in an element exceeds the threshold  $V = 1$ cm/s, we activate  mesh refinement on the fault.
Similarly for the wave fields, if the root means square $v $  exceeds $50$ cm/s  the mesh is refined. Note that on the fault we have two levels of mesh refinement, see Figure \ref{fig:adaptive_mesh_refinement2}. That is the refined elements on the fault are 9 times smaller than the initial coarse mesh, while the refinement meshes for the wave fields are 3 times smaller than the original coarse mesh.

Snap shots of the particle velocity $v_x$ are shown in figure \ref{fig:adaptive_mesh_refinement}.
Note that initially,  the two  elements closest the hypocenter are refined.
After the nucleation, the rupture progresses along  the fault. The adaptive mesh refinement tracks the rupture front and the accompanying elastic waves.
\begin{figure}[h!]
\begin{center}
	\includegraphics[draft=false,width=0.5\columnwidth]{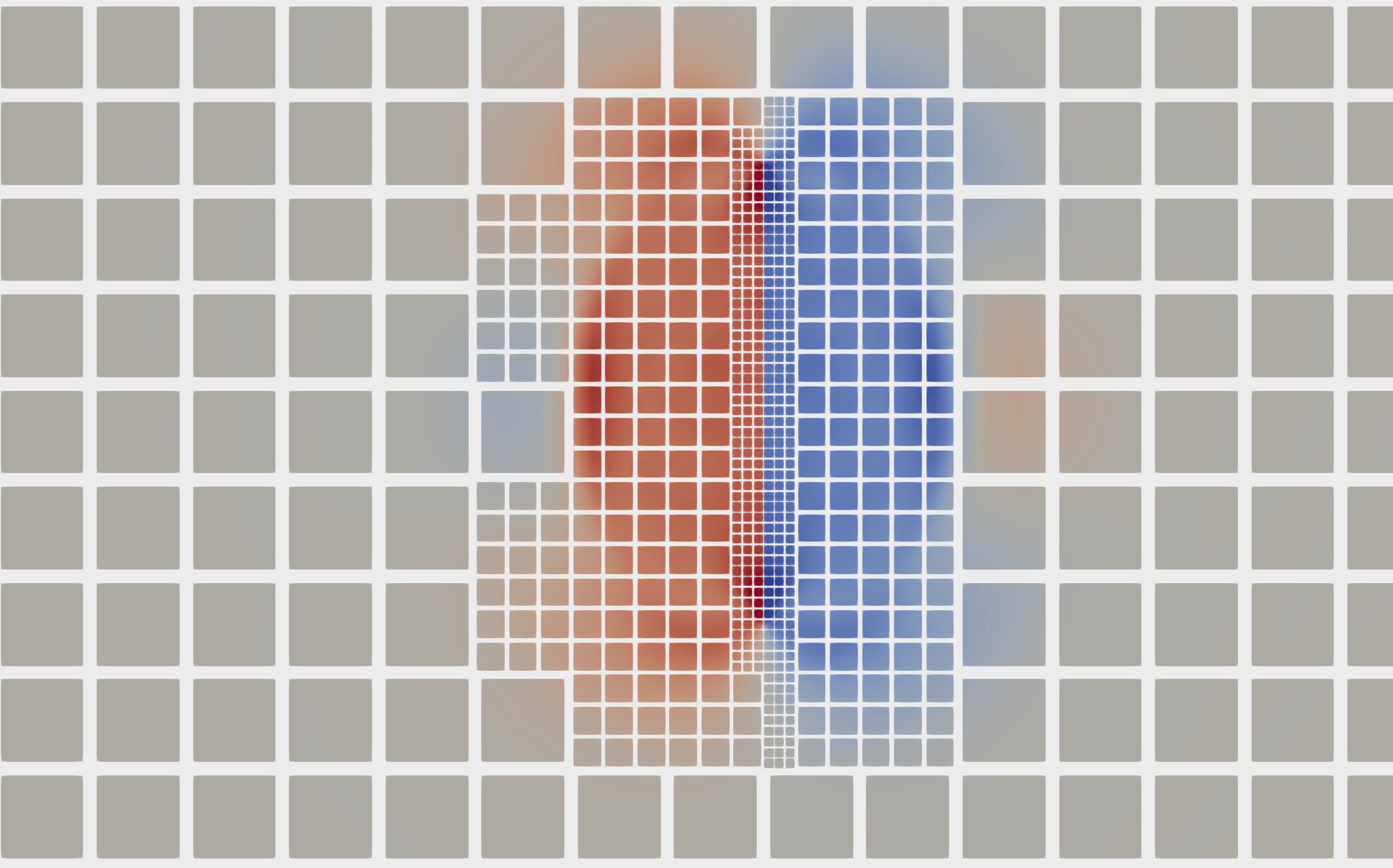} 
\end{center}
\caption{A snapshot of the particle velocity $v_y$  at $t = 1~$s zoomed closer to the fault, showing multiple levels of mesh refinement. }
\label{fig:adaptive_mesh_refinement2}
 \end{figure}
\section{Summary and outlook}
We have developed a new DG method approximation of the linear elastic wave equation incorporating physical interface and boundary conditions acting at element boundaries. 
Our original idea is to use friction to glue DG elements together, in an elastic solid, in a provably stable manner. Thus, all DG inter-element interfaces are frictional interfaces with associated frictional strength.  Classical inter-element interfaces where slip is not permitted have infinite frictional strength, and can never be broken by any load of finite magnitude. Other weak interfaces where frictional  slip can be accommodated have  finite nonlinear frictional strength, and are governed by a generic  nonlinear friction law \cite{Scholz1998, Rice1983,  JRiceetal_83,  Andrews1985}.  External boundaries of the domain are closed with a general linear well-posed and energy-stable   boundary conditions, modeling various geophysical phenomena.    

Our new physics based numerical flux is compatible with all well-posed boundary and interface conditions. By construction our flux implementation is upwind and yields energy identity analogous to the continuous energy estimate. To begin with, our analysis here focuses on a 1D model problem, but the results have  been extended  to multiple space dimensions and complex geometries, and will reported in our forthcoming paper. We present numerical experiments to demonstrate numerical stability, higher order accuracy and optimal convergence rate, for polynomial degree $N\le 10$. Further,  2D numerical examples are  presented to demonstrate the extension of  our method to multiple space  dimensions,  make comparisons with the Rusanov flux and to show the robustness of our method.

The code, as a Jupyter Python Notebook, for the 1D model problem, is publicly available on \emph{Seismolive} \cite{Krischer2018}\\
 ({http://seismo-live.org/}), an online educational  software for computational seismology. The method has been extended to 3D \cite{SIStag2019}, and implemented in \emph{ExaHyPE} \cite{ExaHyPE2019}, a simulation engine for hyperbolic PDEs, on adaptive Cartesian meshes, for exa-scale supercomputers. This software, \emph{ExaHyPE},  is open source: https://exahype.eu/exahype-engine.

\section*{Acknowledegments}
The work presented in this paper was enabled by funding from the European Union's Horizon 2020 research and innovation program under grant agreement No 671698 (ExaHyPE). 
\newline
{}\hfill{\includegraphics[angle=90, width=0.15\textwidth]{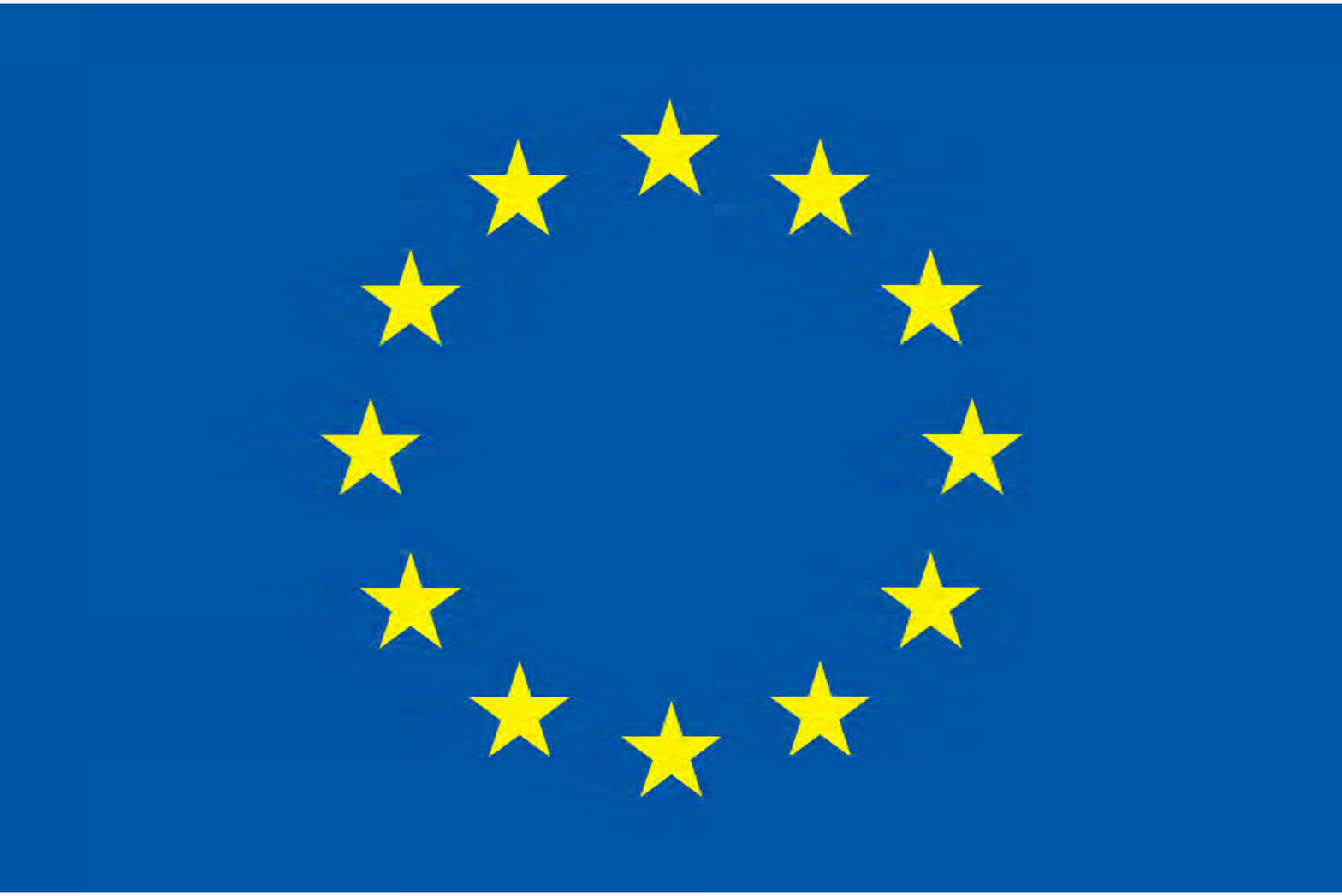}}

A.-A.G.\ acknowledges additional support by the German Research Foundation (DFG) (projects no.~KA 2281/4-1, GA 2465/2-1, GA 2465/3-1), by BaCaTec (project no.~A4) and BayLat, by KONWIHR -- the Bavarian Competence Network for Technical and Scientific High Performance Computing (project NewWave), by KAUST-CRG (GAST, grant no.~ORS-2016-CRG5-3027 and FRAGEN, grant no.~ORS-2017-CRG6 3389.02), by the European Union's Horizon 2020 research and innovation program (ChEESE, grant no.~823844).

\appendix


\end{document}